\documentclass{amsart}
\usepackage{comment}
\usepackage{enumerate,amssymb,amsmath,graphics,epsfig}
\usepackage{color}
\usepackage{xspace}
\usepackage{mathrsfs}
\usepackage{arydshln}
\usepackage[labelfont=rm,font=small,labelformat=simple]{subcaption}

\usepackage{parskip}
\usepackage{pslatex}
\usepackage{amsfonts}

\usepackage{tikz}
\usepackage{lmodern}
\usepackage{graphicx}
\usepackage{bm}
\usepackage{setspace}
\usepackage{hyperref}
\usepackage{caption}
\usepackage{subcaption}
\usepackage{marginnote}
\usepackage[foot]{amsaddr}

\newcommand\gX{\mathcal{X}}
\newcommand\gA{\tilde{\mathcal{A}}}
\newcommand\gB{\tilde{\mathcal{B}}}
\newcommand\gU{\mathcal{U}}
\newcommand\gV{\mathcal{V}}
\newcommand\gT{\mathcal{T}}
\newcommand\gF{\mathcal{F}}
\newcommand\gG{\mathcal{G}}
\newcommand\gE{\mathcal{E}}
\newcommand\glambda{\xi}
\newcommand\RE{\mathbb{R}}
\newcommand\CO{\mathbb{C}}
\newcommand\Tlam{\tilde{\kappa}}

\theoremstyle{plain}
\newtheorem{theorem}{Theorem}
\newtheorem{proposition}[theorem]{Proposition}

\theoremstyle{remark}

\newtheorem{property}{Property}
\newtheorem{remark}{Remark}

\makeatletter
\renewcommand{\email}[2][]{%
  \ifx\emails\@empty\relax\else{\g@addto@macro\emails{,\space}}\fi%
  \@ifnotempty{#1}{\g@addto@macro\emails{\textrm{(#1)}\space}}%
  \g@addto@macro\emails{#2}%
}
\makeatother

\begin{document}
\title[]
{On the spectrum of an operator associated with least-squares finite elements
for linear elasticity}

\author{Linda Alzaben~$^{\ast}$}
\author{Fleurianne Bertrand~{$^\star$}}
\author{Daniele Boffi~$^\dagger$}
\address[$\ast$]{King Abdullah University of Science and Technology (KAUST), Saudi Arabia}
\address[$\star$]{Faculty of Electrical Engineering, Mathematics and Computer Science, University of Twente}
\address[$\dagger$]{King Abdullah University of Science and Technology (KAUST), Saudi Arabia and Universit\`a degli Studi di Pavia, Italy}
\email[$\ast$]{linda.alzaben@kaust.edu.sa}
\email[$\star$]{f.bertrand@utwente.nl}
\email[$\dagger$]{daniele.boffi@kaust.edu.sa}

\begin{abstract}
In this paper we provide some more details on the numerical analysis and we
present some enlightening numerical results related to the spectrum of a
finite element least-squares approximation of the linear elasticity
formulation introduced recently. We show that, although the formulation is
robust in the incompressible limit for the source problem, its spectrum is
strongly dependent on the Lam\'e parameters and on the underlying mesh.
\end{abstract}
\maketitle

\section{Introduction}

In this paper we continue the discussion started in~\cite{FB_eig,FB_LE} about
the spectrum of operators arising from least-squares finite element
approximation of partial differential equations.
In~\cite{FB_eig} several least-squares formulations associated with the Laplace
problem and in~\cite{FB_LE} two formulations associated with the linear
elasticity problem were considered. Here we continue the analysis of the
two-field formulation presented in~\cite{FB_LE}. The two-field formulation was
introduced in~\cite{Cai} for the approximation of the source problem and has
the merit of providing a robust discretization also when the system approaches
the incompressible limit.

Several methods are available for the computation of the eigenvalues and
eigefunctions in linear elasticity and the aim of this paper is not to compete
with them. However, a good knowledge of the properties of the spectrum is
useful for several applications. For instance, if a transient problem is
approximated by using the two-field approach for the space
semi-discretization, then the behavior of the solution depends essentially on
the discrete eigenvalues of the operator corresponding to the least-squares
model. Hence, we believe that a study of the approximation of the eigenvalues
and the eigenfunctions may be interesting in view of a better understanding of
the scheme under investigation.

The aim of this paper is two-fold. First, we complete the analysis presented
in~\cite{FB_LE} where only a sketch of the main ideas were indicated.
Actually, since the least-squares formulation is not symmetric, the
convergence analysis should take into account the dual problem for which we
provide a careful description in this paper.
Second, we present a bunch of numerical experiments that highlight some
properties of the spectrum. A peculiarity of our operators is that the
continuous problem has positive and real eigenvalues, while its approximation
may have eigenvalues everywhere in the complex plane.
Our numerical tests show that for small values of the Lam\'e constant
$\lambda$ (i.e., when the considered elastic solid is far from being
incompressible) the discrete spectrum is generally \emph{well-behaving} that
is, it is distributed in the right half of the complex plane and has a small
imaginary part. On the other hand, when the solid tends to the incompressible
limit, as $\lambda$ increases, the distribution of the discrete spectrum is
more spread in the entire complex plane, including eigenvalues with negative
real part. We present several examples of this behavior and discuss how it may
depend on the chosen mesh sequence. We conclude that, although the two-field
formulation has been proved to be robust for the approximation of the source
problem corresponding to linear elasticity, the spectrum of the discrete
solution operator is \emph{more stable} when a small value of $\lambda$ is
considered.

It will be interesting in the future to discuss in more detail the asymptotic
exactness of the least-squares operator associated to linear elasticity proved
in~\cite{cc}, and to see how it is related to eigenvalue problems and to the
situation presented in this paper.

In Section~\ref{sc:Contiuouse_porblem} we describe the first order system
corresponding to linear elasticity and recall its two-field least-squares
representation. Section~\ref{se:gal} deals with its finite element
discretization and Section~\ref{se:convergence analysis} develops the
convergence analysis. Finally, we report our numerical results in
Section~\ref{se:num}.

\section{The continuous problem}
\label{sc:Contiuouse_porblem}
Consider a polytopal domain $\Omega$ in $\mathbb{R}^d$ $(d = 2 , 3)$. 
We partition the boundary of the domain into two open subsets 
$\Gamma_D$ and $\Gamma_{N}$ such that 
$\partial\Omega = \overline{\Gamma}_{D}\cup \overline{\Gamma}_{N}$ and 
$\Gamma_{D}\cap \Gamma_{N} = \emptyset$ with $\mathbf{n}$ 
denoting the outward unit vector normal to $\Gamma_{N}$. 
The linear elasticity problem that we are going to consider consists of
finding a stress 
tensor $\underline{\bm{\sigma}} = (\sigma_{i,j})_{d \times d}$
and a displacement field $\mathbf{u} =(u_1,\hdots,u_d)^\top$ for a given 
body force $\mathbf{f} = (f_1, \hdots, f_d)^\top$ that satisfy the system 
\begin{equation}\label{eq:elasticity_First_order_system}
\begin{cases}
\underline{\bm{\sigma}} - 
\mathcal{C}\underline{\bm{\epsilon}}(\mathbf{u}) 
=
\underline{\bf{0}} &\quad \text{in} \quad \Omega, \\
\bf{div}\underline{\bm{\sigma}} =
- \mathbf{f}
&\quad \text{in} \quad \Omega, \\
\mathbf{u} = 
\bf{0} \quad &\quad \text{on} \quad \Gamma_{D}, \\
\mathbf{n} \cdot \underline{\bm{\sigma}} = 
\bf{0} &\quad \text{on} \quad{\Gamma_{N}},
\end{cases}
\end{equation}
where $\underline{\bm{\epsilon}}(\mathbf{v})$ is
the symmetric gradient (also known as the strain tensor) given by
$$\underline{\bm{\epsilon}}({\mathbf{v}}) = 
\frac{1}{2}(\underline{\bm\nabla}\mathbf{v} + 
(\underline{\bm\nabla}\mathbf{v})^\top),$$
and $\mathcal{C}$ is the elasticity tensor (a symmetric operator) 
for isotropic, homogeneous material 
defined in terms of the Lam\'e constants 
$\mu$ and $\lambda$ as
$$ \mathcal{C} \underline{\bm{\epsilon}}({\mathbf{v}}) =
\lambda \text{tr}(\underline{\bm{\epsilon}}({\mathbf{v}}))\underline{\mathbf{I}} + 2\mu \underline{\bm{\epsilon}}({\mathbf{v}}),$$
with $\underline{\mathbf{I}}$ being the identity tensor.

It is well known that there is a relation between stress and strain 
tensors such that:
\begin{equation}\label{eq:starin_stress_relation}
\underline{\bm{\tau}} = 
\mathcal{C}\underline{\bm{\epsilon}}(\mathbf{v}) \quad
\text{or} \quad
\underline{\bm{\epsilon}}(\mathbf{v}) = 
\mathcal{A}\underline{\bm{\tau}},
\end{equation}
where $\mathcal{A}$ is the compliance tensor given by 
\begin{equation}\label{eq:compliance_tensor}
\mathcal{A}\underline{\bm{\tau}} =
\frac{1}{2\mu}\bigg(\underline{\bm{\tau}} -
\frac{\lambda}{d\lambda + 
2\mu}\text{tr}(\underline{\bm\tau})\underline{\mathbf{I}}\bigg).
\end{equation}
If $\lambda$ is finite, then from \eqref{eq:starin_stress_relation} we have
$\mathcal{A} = \mathcal{C}^{-1}$.
However, as $\lambda$ approaches $\infty$ the elasticity tensor 
blows up and the material becomes nearly incompressible or incompressible.
For this reason it may be more convenient to use the relation 
$\underline{\bm{\epsilon}}(\mathbf{v}) = \mathcal{A}{\underline{\bm{\tau}}}$
if we are interested in formulations that are stable uniformly in $\lambda$.
We can then rewrite System~\eqref{eq:elasticity_First_order_system} as:
find a symmetric $d\times d$ stress tensor
$\underline{\bm\sigma}$ and a displacement vector-field $\mathbf{u}$ such that
\begin{equation}\label{eq:compliance_First_order_system}
\begin{cases}
\mathcal{A}\underline{\bm{\sigma}} - 
\underline{\bm{\epsilon}}(\mathbf{u}) =
\underline{\bm{0}} &\quad \text{in} \quad \Omega, \\
\bf{div}\underline{\bm{\sigma}} =
- \mathbf{f} 
&\quad \text{in} \quad \Omega, \\
\mathbf{u} = 
\bm{0} \quad &\quad \text{on} \quad \Gamma_{D}, \\
\mathbf{n} \cdot \underline{\bm{\sigma}} = 
\bm{0} &\quad \text{on} \quad{\Gamma_{N}}.
\end{cases}
\end{equation} 
%

We are interested in the spectrum of the solution operator associated
with~\eqref{eq:compliance_First_order_system}, hence, we replace the source
term $\mathbf{f}$ by $\omega \mathbf{u}$, where $\omega$ is the eigenvalue
associated with $\mathbf{u}$.
Since our problem is symmetric, we are looking for real eigenvalues
$\omega \in \mathbb{R}$ such that for non vanishing $\mathbf{u}$ and some
$\underline{\bm{\sigma}}$ the following set of equations is satisfied
\begin{equation}\label{eq:eigenvalue_compliance_First_order_system}
\begin{cases}
\mathcal{A}\underline{\bm{\sigma}} - 
\underline{\bm{\epsilon}}(\mathbf{u}) =
\underline{\bm{0}} &\quad \text{in} \quad \Omega, \\
\bf{div}\underline{\bm{\sigma}} = 
- \omega \mathbf{u} 
&\quad \text{in} \quad \Omega, \\
\mathbf{u} = 
\bm{0} \quad &\quad \text{on} \quad \Gamma_{D}, \\
\mathbf{n} \cdot \underline{\bm{\sigma}} = 
\bm{0} &\quad \text{on} \quad{\Gamma_{N}}.
\end{cases}
\end{equation} 
The eigenvalue Problem~\eqref{eq:eigenvalue_compliance_First_order_system} is compact due 
to the regularity properties of the solution of System \eqref{eq:compliance_First_order_system}.
Therefore, the eigenvalues are real and positive and form an increasing sequence 
$$0 < \omega_1\leq \omega_2\leq\cdots\leq\omega_i\leq\cdots  
\quad \text{with}\quad  
\lim_{i\to\infty}\omega_i = \infty,$$
repeated according to their multiplicity so that each eigenvalue $\omega_i$ corresponds 
to a one dimensional eigenspace  $E_i$.

We rewrite~\eqref{eq:compliance_First_order_system} following
the least-squares principle approach introduced in~\cite{Cai}.
As stated by the authors and as a consequence of the above considerations,
this approach has the important advantage of automatically stabilizing the 
stress-dis\-place\-ment system in the incompressible limit when comparing it to 
other approaches, thus being more robust.
The so called two-field formulation (in the unknowns $\mathbf{u}$ and $\underline{\bm{\sigma}}$) 
was achieved by applying the $L^{2}$ norm and minimizing the functional 
\begin{equation}\label{eq:functional}
\mathcal{F}(\underline{\bm\tau},\textbf{v};\textbf{f})
= 
||\mathcal{A}\underline{\bm{\tau}} 
- 
\underline{\bm{\epsilon}}(\textbf{v})||_{0}^{2} +
||\textbf{div}\underline{\bm{\tau}} + \textbf{f} ||_{0}^{2},
\end{equation}
in $\underline{\bm{X}}_N \times H^1_{0,D}(\Omega)^d$, where
\begin{equation*}
\underline{\bm{X}} =
\begin{cases}
\textbf{H}(\textbf{div};\Omega)^d &\quad \text{if} 
\quad\Gamma_N\neq \emptyset,\\
\{\underline{\bm{\tau}}\in\textbf{H}(\textbf{div};\Omega)^d:
\int_{\Omega}\text{tr}(\underline{\bm{\tau}})\text{d}\textbf{x} = 
0\} &\quad \text{if} \quad\Gamma_N= \emptyset,\\
\end{cases}
\end{equation*}
and $\underline{\bm{X}}_N$ being its subspace
\begin{equation*}
\underline{\bm{X}}_N =\{ \underline{\bm{\tau}} \in \underline{\bm{X}}: 
\textbf{n} \cdot \underline{\bm{\tau}} = \textbf{0}\quad \text{on}
\quad \Gamma_N\}.
\end{equation*}
Applying the same strategy
to~\eqref{eq:eigenvalue_compliance_First_order_system} would lead to a
non-linear problem.  Following the original idea developed in~\cite{FB_eig}
for the Laplace operator, in~\cite{FB_LE} the spectrum of the operators
associated with the least-squares $\it{source}$ formulation was considered.
\subsection{The variational formulation.}
The minimization of the functional corresponding to the source problem
in~\eqref{eq:functional}, gives rise to the variational problem: 
find $(\underline{\bm{\sigma}},\textbf{u}) 
\in \underline{\textbf{X}}_N\times H^1_{0,D}(\Omega)^d$ such that
\begin{equation}\label{eq:variationa_form_source}
\begin{cases}
(\mathbf{\mathcal{A}}\underline{\bm{\sigma}},
\mathbf{\mathcal{A}}\underline{\bm{\tau}})
+
(\textbf{div}\underline{\bm{\sigma}},\textbf{div}\underline{\bm{\tau}})
-
(\mathbf{\mathcal{A}}\underline{\bm{\tau}},\underline{\bm{\epsilon}}(\textbf{u}))
=
-(\textbf{f},\textbf{div}\underline{\bm{\tau}})
&\quad \forall \underline{\bm{\tau}} \in \underline{\bm{X}}_N, \\
-
(\mathbf{\mathcal{A}}\underline{\bm{\sigma}}, 
\underline{\bm{\epsilon}}(\textbf{v}))
+
(\underline{\bm{\epsilon}}(\textbf{u}),  
\underline{\bm{\epsilon}}(\textbf{v}))
=
\bm{0}
&\quad \forall\textbf{v} \in H^{1}_{0,D}(\Omega)^d,
\end{cases}
\end{equation}
so that the eigenvalue variational problem associated with the two-field
formulation reads:
find 
$(\omega,\mathbf{u}) \in \mathbb{R}\times H^{1}_{0,D}(\Omega)^d$  with $\textbf{u}\neq 0$ such that for some $\underline{\bm{\sigma}} \in \underline{\bm{X}}_N $ we have
\begin{equation}\label{eq:variationa_form_eig}
\begin{cases}
(\mathbf{\mathcal{A}}\underline{\bm{\sigma}},
\mathbf{\mathcal{A}}\underline{\bm{\tau}})
+
(\textbf{div}\underline{\bm{\sigma}},
\textbf{div}\underline{\bm{\tau}})
-
(\mathbf{\mathcal{A}}\underline{\bm{\tau}},
\underline{\bm{\epsilon}}(\textbf{u}))
=
-\omega (\textbf{u},\textbf{div}\underline{\bm{\tau}})
&\quad \forall \underline{\bm{\tau}} \in \underline{\bm{X}}_N, \\
-
(\mathbf{\mathcal{A}}\underline{\bm{\sigma}}, 
\underline{\bm{\epsilon}}(\textbf{v}))
+
(\underline{\bm{\epsilon}}(\textbf{u}), 
\underline{\bm{\epsilon}}(\textbf{v}))
=
\bm{0}
&\quad \forall\textbf{v} \in H^{1}_{0,D}(\Omega)^d.
\end{cases}
\end{equation} 

As discussed in~\cite{FB_LE,FB_eig}, problem~\eqref{eq:variationa_form_eig}
has the following non symmetric structure
\begin{equation}\label{eq:matrix_structure}
\begin{pmatrix} 
A & B^\top \\ 
B & C 
\end{pmatrix}
\begin{pmatrix} 
x \\ 
y
\end{pmatrix}
= \omega
\begin{pmatrix} 
0 & D\\ 0 & 0 
\end{pmatrix}
\begin{pmatrix} 
x \\ 
y
\end{pmatrix},
\end{equation}
where the operators are associated to the bilinear forms as follows
\begin{equation}\label{eq:the_forms}
\begin{cases}
A : (\mathbf{\mathcal{A}}\underline{\bm{\sigma}},
\mathbf{\mathcal{A}}\underline{\bm{\tau}})
+
(\textbf{div}\underline{\bm{\sigma}},
\textbf{div}\underline{\bm{\tau}}),\\
B : -(\mathbf{\mathcal{A}}\underline{\bm{\sigma}}, 
\underline{\bm{\epsilon}}(\textbf{v})),\\
C :(\underline{\bm{\epsilon}}(\textbf{u}),  
\underline{\bm{\epsilon}}(\textbf{v})),\\
D :-(\textbf{u},\textbf{div}\underline{\bm{\tau}}),
\end{cases}
\end{equation}
such that $x$ and $y$ are associated to $\underline{\bm{\sigma}}$ and $\textbf{u}$, 
respectively, in an abstract setting.
Compared to the Laplacian case in~\cite{FB_eig}, we can
see that the operators $B^\top$ and $-D$ are different from each other, being
related to the bilinear forms 
$-(\mathbf{\mathcal{A}}\underline{\boldsymbol{\tau}},
\underline{\boldsymbol{\epsilon}}(\textbf{u}))$ 
and 
$(\textbf{u},\textbf{div}\underline{\boldsymbol{\tau}})$. In particular, it is
not possible to show that~\eqref{eq:matrix_structure} corresponds to a
symmetric problem.
Thus, we might expect complex eigenvalues among the approximation of
the solutions of~\eqref{eq:variationa_form_eig} even if the eigenvalues of the
continuous problem are real.
This is different from the structure of the FOSLS of the Poisson equation
in~\cite{FB_eig} which corresponds to a symmetric problem when applying the
integration by part to the bilinear form associated to $-D$
(see~\cite[Sec.\ 2.1]{FB_eig}).
\section{Galerkin discretization}
\label{se:gal}

The discrete variational formulation associated with \eqref{eq:variationa_form_eig} is 
established by introducing finite dimensional subspaces 
$\Sigma_h \subset \underline{\bm{X}}_N$, $U_h\subset H^{1}_{0,D}(\Omega)^d$
and by considering discrete variables
$\underline{\bm{\sigma}}_h\subset\Sigma_h$ and $\textbf{u}_h\subset U_h$,
respectively.
Then the Galerkin approximation of the eigenvalue problem is: find 
$(\omega_h, \textbf{u}_h) \in \mathbb{C}\times U_h$ with a non vanishing 
$\textbf{u}_h$ and some $\underline{\bm{\sigma}}_h\in \Sigma_h$ such that
\begin{equation}\label{eq:disdrete_variationa_form_eig}
\begin{cases}
(\mathbf{\mathcal{A}}\underline{\bm{\sigma}}_h,\mathbf{\mathcal{A}}\underline{\bm{\tau}})
+
(\textbf{div}\underline{\bm{\sigma}}_h,\textbf{div}\underline{\bm{\tau}})
-
(\mathbf{\mathcal{A}}\underline{\bm{\tau}},\underline{\bm{\epsilon}}(\textbf{u}_h))
=
-\omega_h (\textbf{u}_h,\textbf{div}\underline{\bm{\tau}})
&\quad \forall \underline{\bm{\tau}} \in \Sigma_h,\\
-
(\mathbf{\mathcal{A}}\underline{\bm{\sigma}}_h, \underline{\bm{\epsilon}}(\textbf{v}))
+
(\underline{\bm{\epsilon}}(\textbf{u}_h), 
\underline{\bm{\epsilon}}(\textbf{v}))
=
\bm{0}
&\quad \forall\textbf{v} \in U_h.
\end{cases}
\end{equation} 
This discrete eigenvalue problem has the same structure as in 
\eqref{eq:matrix_structure} 
with the natural definition of matrices that are associated with 
the bilinear forms defined in \eqref{eq:the_forms}.
Looking at the problem from an algebraic point of view, the solution of this problem
satisfies some properties that are discussed in~\cite{FB_LE,FB_eig}.
In general, our discrete eigenvalue problem in \eqref{eq:disdrete_variationa_form_eig} has the form of a
generalized eigenvalue problem
\begin{equation}\label{eq:generalized_eigenvalue_prob}
\mathcal{M}x = \omega\mathcal{N}x,
\end{equation}
where in our framework the matrix $\mathcal{M}$ is symmetric and invertible,
while $\mathcal{N}$ is non-symmetric and singular.
The aim of this paper is not to discuss how to solve efficiently our problem.
We just indicate that a possible way to resolve the singularity of the system
is to switch the roles of the two matrices $\mathcal{M,N}$ and to consider the
solutions $(\gamma,x)$ of the problem 
$$\mathcal{N}x=\gamma\mathcal{M}x.$$
If $\gamma=0$ we will say that the corresponding eigenvalue
of~\eqref{eq:generalized_eigenvalue_prob} is $\omega=\infty$.
The remaining finite eigenmodes are $(\omega,x)$ where
$\omega= \frac{1}{\gamma}$. \\
We report the following proposition from~\cite{FB_LE} that classifies the eigenvalues 
of our problem.
\begin{proposition}\label{pro:eigval_families}
Consider the matrices associated with the bilinear forms defined in \eqref{eq:the_forms}.
Then the following generalized eigenvalue problem 
\begin{equation}\label{eq:discte_matrix_structure}
\begin{pmatrix} 
A & B^\top \\ 
B & C 
\end{pmatrix}
\begin{pmatrix} 
\hat{\bm{\sigma}}_h \\ 
\hat{\normalfont\textbf{u}}_h
\end{pmatrix}
= \omega_h
\begin{pmatrix} 
0 & D\\ 0 & 0 
\end{pmatrix}
\begin{pmatrix} 
\hat{\bm{\sigma}}_h \\ 
\hat{\normalfont\textbf{u}}_h
\end{pmatrix}
\end{equation}{}
has three families of eigenvalues, specifically:
\begin{itemize}
\item $\omega_h = \infty$ with multiplicity equal to $\text{dim}(\Sigma_h)$ 
\item $\omega_h = \infty$ with multiplicity equal to $\text{dim(ker(D))}$ if $D$ is not full rank
\item $\omega_h =\text{complex eigenvalues}$ counted with their multiplicity
equal to rank($D$)
\end{itemize}
\end{proposition}

We are actually interested in the eigenpairs corresponding to a non-zero 
displacement $\textbf{u}_h\ne\textbf{0}$. This can be easily identified by considering
the following Schur complement approach.
We can rewrite the algebraic linear system of the block matrices 
in \eqref{eq:discte_matrix_structure} as follows:
\begin{equation}\label{eq:Schur_complement}
\begin{cases}
A\bm{\hat{\sigma}}_h + B^\top \hat{\textbf{u}}_h 
= \omega_h D \hat{\textbf{u}}_h, \\
B\hat{\bm{\sigma}}_h + C \hat{\textbf{u}}_h = \textbf{0}.
\end{cases} 
\end{equation}
Looking at the pair of interest, and using the fact that $A$ is invertible,
we can rewrite the first equation in \eqref{eq:Schur_complement} as follows:
$$\hat{\bm{\sigma}}_h = A^{-1}(\omega_h D -B^\top)\hat{\textbf{u}}_h.$$
Now by substituting the value of $\hat{\bm\sigma}_h$ in the second equation of System
\eqref{eq:Schur_complement} and by combining common terms together, we arrive
to the Schur complement linked only to the displacement given by
\begin{equation}\label{eq:Shur_u}
(B A^{-1}B^\top - C)\hat{\textbf{u}}_h = \omega_h B A^{-1} D \hat{\textbf{u}}_h.
\end{equation}
\begin{remark}\label{rem:Shur_sigma}
In general, one can deduce two Schur complements associated 
with System \eqref{eq:Schur_complement}, one linked with the 
$\hat{\bm{\sigma}}_h$ variable and another to the variable
$\hat{\textbf{u}}_h$.
That is, we can initially seek for the pair $(\omega_h, \hat{\bm{\sigma}}_h)$ and then obtain the eigenpair 
$(\omega_h,\hat{\textbf{u}}_h)$ (see~\cite[Sec.\ 3.1]{LZ_FB_DB}).
If one naively picks the Schur complement associated with $\hat{\bm{\sigma}}_h$, 
artificial displacement (that is $\hat{\textbf{u}}_h = 0$ ) which is 
not related to our problem occurs, thus, wrong solution would then be present. 
Hence, one needs to be careful while selecting the proper Schur complement 
associated to the variable of interest in order to ensure that we are
considering only genuine pairs for the displacement.
\end{remark}
\section{Convergence analysis}\label{se:convergence analysis}

The convergence analysis of eigenvalue problems has the standard abstract setting 
presented in~\cite{acta_num,Babuska_eig_prob}. 
Usually the analysis is divided into two steps: first considering the
convergence of eigenmodes (with the absence of superior modes) that is a
consequence of the convergence of the discrete solution operator to the
continuous one, then discussing the rate of convergence.

To this end we introduce the solution operator 
$T:L^2(\Omega) \rightarrow L^2(\Omega)$ associated with formulation~\eqref{eq:variationa_form_eig}. 
Given 
$\textbf{f} \in L^2(\Omega)^d$ we define $T\textbf{f}\in H_0^1(\Omega)^d$ 
being the second component $\textbf{u}$ of the solution of \eqref{eq:variationa_form_source}, which solves the following problem for some 
$\underline{\bm{\sigma}} \in \underline{\bm{X}}_N$
\begin{equation}\label{eq:cont_operator_problem}
\begin{cases}
(\mathbf{\mathcal{A}}\underline{\bm{\sigma}},\mathbf{\mathcal{A}}\underline{\bm{\tau}})
+
(\textbf{div}\underline{\bm{\sigma}},\textbf{div}\underline{\bm{\tau}})
-
(\mathbf{\mathcal{A}}\underline{\bm{\tau}},\underline{\bm{\epsilon}}(T\textbf{f}))
= 
-(\textbf{f}, \textbf{div}\underline{\bm{\tau}})
&\quad \forall \underline{\bm{\tau}} \in \underline{\bm{X}}_N, \\
-
(\mathbf{\mathcal{A}}\underline{\bm{\sigma}}, \underline{\bm{\epsilon}}(\textbf{v}))
+
(\underline{\bm{\epsilon}}(T\textbf{f}), 
\underline{\bm{\epsilon}}(\textbf{v}))
=
\bm{0}
&\quad \forall\textbf{v} \in H^{1}_{0,D}(\Omega)^d.
\end{cases}
\end{equation}   

Since the range of the operator $T$ is included in the space 
$H_0^1(\Omega)^d$
which is compact in $L^2$, then $T$ is compact.

The discrete solution operator $T_h: L^2(\Omega)\rightarrow L^2(\Omega)$ with
$\textbf{f}\in L^2(\Omega)^d$ is defined as the second component
$T_h\textbf{f}\in U_h$ of the solution of the Galerkin approximation
of~\eqref{eq:variationa_form_source} which solves the following problem for
some $\underline{\bm{\sigma}}_h \in \Sigma_h$
\begin{equation}\label{eq:discrete_operator_problem}
\begin{cases}
(\mathbf{\mathcal{A}}\underline{\bm{\sigma}}_h,\mathbf{\mathcal{A}}\underline{\bm{\tau}})
+
(\textbf{div}\underline{\bm{\sigma}}_h,\textbf{div}\underline{\bm{\tau}})
-
(\mathbf{\mathcal{A}}\underline{\bm{\tau}},\underline{\bm{\epsilon}}(T_h\textbf{f}))
= 
-(\textbf{f}, \textbf{div}\underline{\bm{\tau}})
&\quad \forall \underline{\bm{\tau}} \in \Sigma_h, \\
-
(\mathbf{\mathcal{A}}\underline{\bm{\sigma}}_h, \underline{\bm{\epsilon}}(\textbf{v}))
+
(\underline{\bm{\epsilon}}(T_h\textbf{f}), 
\underline{\bm{\epsilon}}(\textbf{v}))
=
\bm{0}
&\quad \forall\textbf{v} \in U_h.
\end{cases}
\end{equation}

In~\cite{FB_LE} it was stated that the uniform convergence of $T_h$ to $T$
holds true, so that the discrete eigenvalues converge to the continuous ones
without spurious modes.
In the next theorem we will give a rigorous proof of the uniform convergence,
by giving more details than what was previously published.

We start by recalling the expression of the generic dual problem associated with the
variational formulation of the source problem as follows: given
$\textbf{g}\in L^2(\Omega)^d$
find $\underline{\bm{\chi}} \in \underline{\bm{X}}_N$ and $\textbf{p}\in H^1_0(\Omega)^d$ such that
\begin{equation}\label{eq:dual_problem_variational_form}
\begin{cases}
(\mathbf{\mathcal{A}}\underline{\bm{\chi}},\mathbf{\mathcal{A}}\underline{\bm{\xi}})
+
(\textbf{div}\underline{\bm{\chi}},\textbf{div}\underline{\bm{\xi}})
-
(\mathbf{\mathcal{A}}\underline{\bm{\xi}},\underline{\bm{\epsilon}}(\textbf{p}))
= 
\bm{0}
&\quad \forall \underline{\bm{\xi}} \in \underline{\bm{X}}_N, \\
-
(\mathbf{\mathcal{A}}\underline{\bm{\chi}}, \underline{\bm{\epsilon}}(\textbf{v}))
+
(\underline{\bm{\epsilon}}(\textbf{p}),
\underline{\bm{\epsilon}}(\textbf{v}))
=
(\textbf{g}, \textbf{v})
&\quad \forall\textbf{v} \in H^{1}_{0,D}(\Omega)^d.
\end{cases}
\end{equation}

\begin{theorem}\label{thm:uniform_convergence}
Let $s>1/2$ be such that $\normalfont\textbf{u} \in H^{1+s}(\Omega)^d$, where
$\normalfont\textbf{u}$ is the second component of the solution to
\eqref{eq:variationa_form_source}, and such that the following stability of
the dual problem holds true
\begin{equation}\label{bound:minimum_regularity_dual}
||\textbf{p}||_{s+1} + ||\underline{\bm{\chi}}||_s+||\normalfont\textbf{div}\underline{\bm{\chi}}||_{s+1}   \leq
C ||\textbf{g}||_0.
\end{equation}
Let $\normalfont\textbf{u}_h \in U_h$  be the numerical solution corresponding
to $\normalfont\textbf{u}$.
Assume that the finite element spaces $\Sigma_h$ and $U_h$ satisfy the following 
approximation properties 
\begin{equation}
\inf\limits_{\underline{\bm{\tau}} \in\Sigma_h}|| \underline{\bm{\chi}} - 
\underline{\bm{\tau}}||_{\normalfont\textbf{div}}
\leq
C h^s (||\underline{\bm{\chi}}||_s 
+ || \normalfont\textbf{div} \underline{\bm{\chi}}||_{1+s}),
\label{eq:RTapprox}
\end{equation}
\begin{equation}
  \inf\limits_{\textbf{v} \in U_h}|| \textbf{p} - 
\textbf{v}||_{1}
\leq
C h^s ||\textbf{p}||_{1+s},
\label{eq:Papprox}
\end{equation}
then we have
$$
||\normalfont\textbf{u} - \normalfont\textbf{u}_h||_0 \leq C h^s 
||\textbf{f}||_0.
$$
\end{theorem}
\begin{proof}
Consider the dual problem \eqref{eq:dual_problem_variational_form} and recall
the stability bound~\eqref{bound:minimum_regularity_dual}.
By choosing $\textbf{g} = \textbf{u} - \textbf{u}_h$ and taking the test functions in \eqref{eq:dual_problem_variational_form} to be 
$\underline{\bm{\xi}} = \underline{\bm{\sigma}} - \underline{\bm{\sigma}}_h$, $\textbf{v} = \textbf{u}-\textbf{u}_h$, we have
\begin{align*}
||\textbf{u}- \textbf{u}_h||_0^2 &= (\mathbf{\mathcal{A}}\underline{\bm{\chi}},\mathbf{\mathcal{A}}(\underline{\bm{\sigma}} - \underline{\bm{\sigma}}_h)) 
+
(\textbf{div}\underline{\bm{\chi}},\textbf{div}(\underline{\bm{\sigma}} -  
\underline{\bm{\sigma}}_h)) \\
& - 
(\mathbf{\mathcal{A}}(\underline{\bm{\sigma}} - \underline{\bm{\sigma}}_h),
\underline{\bm{\epsilon}}(\textbf{p})) 
-
(\mathbf{\mathcal{A}}\underline{\bm{\chi}},
\underline{\bm{\epsilon}}(\textbf{u}-\textbf{u}_h))\\
& +
(\underline{\bm{\epsilon}}(\textbf{p}), \underline{\bm\epsilon}(\textbf{u}-\textbf{u}_h)).
\end{align*}
Now using the error equations of \eqref{eq:variationa_form_source}, for all 
$\underline{\bm{\tau}}_h \in \Sigma_h$ and $\textbf{v}_h \in U_h$ we arrive to
\begin{align*}
||\textbf{u}- \textbf{u}_h||_0^2 
&= 
(\mathcal{A}(\underline{\bm{\chi}}-\underline{\bm{\tau}}_h),\mathcal{A}(\underline{\bm{\sigma}} - \underline{\bm{\sigma}}_h)) 
+
(\textbf{div}(\underline{\bm{\chi}}-\underline{\bm{\tau}}_h),\textbf{div}(\underline{\bm{\sigma}}- \underline{\bm{\sigma}}_h)) \\
&- 
(\mathbf{\mathcal{A}}(\underline{\bm{\sigma}} - \underline{\bm{\sigma}}_h),
\underline{\bm{\epsilon}}(\textbf{p}-\textbf{v}_h)) 
-
(\mathbf{\mathcal{A}}(\underline{\bm{\chi}}-\underline{\bm{\tau}}_h),
\underline{\bm{\epsilon}}(\textbf{u}-\textbf{u}_h))\\
& +
(\underline{\bm{\epsilon}}(\textbf{p}-\textbf{v}_h), \underline{\bm\epsilon}(\textbf{u}-\textbf{u}_h)).
\end{align*}
After using the Cauchy--Schwarz inequality, the definitions of
$||\cdot||_{\textbf{div}}$, $||\cdot||_1$, and collecting common terms, we arrive to
\begin{align*}
||\textbf{u}-\textbf{u}_h||_0^2 
&\leq 
C (||\underline{\bm{\chi}}-\underline{\bm{\tau}}_h||_\textbf{div}
+
||\textbf{p}-\textbf{v}_h||_1)
(||\underline{\bm{\sigma}}- \underline{\bm{\sigma}}_h||_\textbf{div}+
||\textbf{u}-\textbf{u}_h||_1)\\
&\leq
C (||\underline{\bm{\chi}}-\underline{\bm{\tau}}_h||_\textbf{div}
+
||\textbf{p}-\textbf{v}_h||_1)
||\textbf{f}||_0 .
\end{align*}
Applying the approximation properties of $\Sigma_h$ and $ U_h$ to the above
and using the regularity of the dual problem
in~\eqref{bound:minimum_regularity_dual} we get
$$
||\textbf{u}-\textbf{u}_h||_0^2 \leq Ch^s ||\textbf{u}-\textbf{u}_h||_0 ||\textbf{f}||_0.
$$
\end{proof}

\begin{remark}
The regularity assumed in~\eqref{bound:minimum_regularity_dual} requires some
explanation. Actually, it has been recently observed~\cite{brendan} that a
dual problem arising from least-squares formulations does not necessarily
share the same regularity properties as the original problem. By inspecting
the variational form~\eqref{eq:dual_problem_variational_form} it can be
observed that formally the solution of the dual problem satisfies
the following strong formulation
\begin{equation}\label{eq:dual_strong_form}
\begin{split}
\mathcal{A}\underline{\bm{\chi}}  = \underline{\bm{\epsilon}}(\textbf{p}) -
\mathcal{C}\nabla\textbf{w}\\
-\textbf{div}\underline{\bm{\chi}} = \textbf{w}
\end{split}
\end{equation}
where $\textbf{w}$ satisies
\[
- \textbf{div}(\mathcal{C} \underline{\bm{\epsilon}} (\textbf{w}))
= \textbf{g}.
\]
We can then deduce the existence of a suitable $s$ so
that~\eqref{bound:minimum_regularity_dual} is satisfied. However, we cannot
conclude that if $\textbf{u}$ is more regular, then the same regularity can be
automatically transferred to the solution of the dual problem. We will go back
to this discussion when estimating the rate of convergence of the
approximation of the eigenvalues.
\label{re:duality}
\end{remark}

\begin{remark}
The approximation properties~\eqref{eq:RTapprox} and~\eqref{eq:Papprox} stated
in Theorem~\ref{thm:uniform_convergence} are satisfied by the most natural
finite element spaces that can be used for the approximation of our problem.
For instance Raviart--Thomas finite element enjoy~\eqref{eq:RTapprox} and
standard Lagrangian finite element achieve~\eqref{eq:Papprox}. An interesting
question is whether such estimates depend on the Lam\'e coefficients.
Actually, the constants $C$ are independent of the coefficients, while it
can be expected that the norms of the dual solution might be affected by them.
\end{remark}

The uniform convergence implies the convergence of the discrete eigenvalues
which we recall in the next proposition 
(see~\cite{acta_num} and~\cite{Babuska_eig_prob}).

\begin{proposition}\label{pro:convergence_of_T}
Assume that the operator $T_h$ converges in norm to $T$ as $h$ goes to
zero, that is,
\begin{equation}\label{eq:convergence_of_T}
||T-T_h||_{\mathcal{L}(X)}\to0.
\end{equation}
Let $\Tlam_i=\Tlam_{i+1}=\dots = \Tlam_{i+m-1}$ be an eigenvalue of
multiplicity $m$ of the operator $T$.
Then for $h$ small enough, (so that the total number of discrete finite
eigenvalues is greater than or equal to $i+m-1$), the $m$ discrete eigenvalues
$\Tlam_{j,h}$ ($j= i,\dots, i+m-1$)  of with $T_h$ converge to $\Tlam_j$.
Moreover, let $E = \bigoplus_{j=i}^{i+m-1} E_j$ be the continuous eigenspace spanned by 
$\{u_i,\dots,u_{i+m-1}\}$, and $E_h = \bigoplus_{j=i}^{i+m-1} E_{j,h}$ be the
direct sum of the corresponding discrete eigenspaces spanned by 
$\{u_{i,h},\dots,u_{i+m-1,h}\}$.
Then the corresponding eigenfunctions converge, that is
\begin{equation}
\delta(E,E_{h})\to0,
\end{equation}
where $\delta$ denotes the gap between Hilbert subspaces.
\end{proposition}
Moving to the rate of convergence, we start by stating the classical result
for the convergence of the eigenfunctions (see, for instance,~\cite[Thm.\
7.1]{Babuska_eig_prob}).

\begin{theorem}
Let $X$ be $L^2(\Omega)^d$ or $H^1(\Omega)^d$.
With the same notation and assumptions of
Proposition~\ref{pro:convergence_of_T} we have that
\[
\delta(E,E_{h})\leq  C ||(T-T_h)_{|E}||_{\mathcal{L}(X)}
\]
where the gap $\delta$ is evaluated in the norm of $X$.
\label{th:eigfun}
\end{theorem}

The estimation of the rate of convergence for the eigenvalues requires a more
careful analysis since the discrete eigenvalue problem is not symmetric.
To this aim, we embed our problem in the framework of general variationally
posed eigenvalue problems as follows. A general variationally posed eigenvalue
problem reads: given a Hilbert space $\gX$ and suitably defined bilinear forms
$\gA:\gX\times\gX\to\RE$ and $\gB:\gX\times\gX\to\RE$, find
eigenvalues $\glambda\in\CO$ and non vanishing eigenfunctions $\gU\in\gX$ such
that
\[
\gA(\gU,\gV)=\glambda\gB(\gU,\gV)\qquad\forall\gV\in\gX.
\]
We then introduce the solution operator $\gT:\gX\to\gX$ associated with the
general formulation so that if $\gF\in\gX$ then $\gT\gF\in\gX$ is the solution
of the problem
\[
\gA(\gT\gF,\gV)=\gB(\gF,\gV)\qquad\forall\gV\in\gX.
\]
For the moment we proceed formally and assume that the operator $\gT$ is well
defined. It is out of the scope of this paper to recall all the detail of the
general theory which we are going to apply to our specific formulation.
We identify $\gX$ with its dual so that the adjoint operator $\gT^*:\gX\to\gX$
is associated with the dual problem, that is, given $\gG\in\gX$ we have that
$\gT^*\gG\in\gX$ is equal to the solution of the problem
\[
\gA(\gV,\gT^*\gG)=\gB(\gV,\gG)\qquad\forall\gV\in\gX.
\]

Given a finite element space $\gX_h\subset\gX$, we can consider the discrete
eigenvalue problem: find eigenvalues $\glambda_h\in\CO$ and non vanishing
eigenfunctions $\gU_h\in\gX_h$ such that
\[
\gA(\gU_h,\gV)=\glambda_h\gB(\gU_h,\gV)\qquad\forall\gV\in\gX_h.
\]
Analogously, we can consider the discrete solution operator $\gT_h:\gX\to\gX$
that satisfies $\gT_h\gF\in\gX_h$ and
\[
\gA(\gT_h\gF,\gV)=\gB(\gF,\gV)\qquad\forall\gV\in\gX_h,
\]
for $\gF\in\gX$ and its adjoint $\gT^*_h:\gX\to\gX$ which satisfies
$\gT_h^*\gG\in\gX_h$ and
\begin{equation}
\gA(\gV,\gT^*_h\gG)=\gB(\gV,\gG)\qquad\forall\gV\in\gX_h,
\label{eq:sym}
\end{equation}
for $\gG\in\gX$.

We assume that the eigenmodes of $\gT_h$ are approximating correctly the ones
of $\gT$ and we discuss the rate of convergence.
We consider now an eigenvalue
$\glambda_i=\glambda_{i+1}=\dots=\glambda_{i+m-1}$ of $\gT$ of multiplicity
$m$ and the corresponding discrete eigenvalues $\glambda_{j,h}$
($j=i,\dots,i+m-1$) of $\gT_h$ approximating it. We denote by $\gE_j$ the
eigenspace of $\gT$ corresponding to $\glambda_j$ and by $\gE_j^*$ the one
of $\gT^*$.
The following theorem summarizes the rate of convergence of the discrete
eigenvalues (see, for instance,~\cite[Thm.\ 7.3]{Babuska_eig_prob}).

\begin{theorem}\label{thm:bound_eigvals_with_T_T*}
Let $\{\gU_i,\dots,\gU_{i+m-1}\}$ form a basis for the eigenspace
$\gE=\bigoplus_{j=i}^{i+m-1}\gE_j$ and let $\{\gU^*_i,\dots,\gU^*_{i+m-1}\}$
be the dual basis of the corresponding eigenspace
$\gE^*=\bigoplus_{j=i}^{i+m-1}\gE^*_j$ associated with $\gT^*$.
Then for $j = i,\dots, i+m-1$,
\begin{equation}
|\glambda_j-\glambda_{j,h}|\le C\left\{\sum_{k,\ell=1}^m|((\gT-\gT_h)\gU_k,\gU_\ell^*)|
+\|(\gT-\gT_h)_{|\gE}\|_{\mathcal{L}(\gX)}\|(\gT^*-\gT^*_h)_{|\gE^*}\|_{\mathcal{L}(\gX)}\right\}.
\end{equation} 
\end{theorem}

\begin{remark}
In the general case of non symmetric problems, the previous theorem should
take into account also the ascent multiplicities of the eigenvalues.
In our case, however, the continuous problem is symmetric so that we know that
the ascent multiplicities of our eigenvalues are always equal to one.
\end{remark}

We now conclude this section by showing how the abstract results apply to our
specific problem.
As noted in Section~\ref{sc:Contiuouse_porblem}, the eigenvalues $\omega_i$ of
the continuous problem~\eqref{eq:eigenvalue_compliance_First_order_system} are
real, positive and form an increasing sequence with each being repeated
according to its multiplicity. 
One the contrary, the discrete eigenvalues $\omega_{i,h}$
of~\eqref{eq:disdrete_variationa_form_eig} are complex and can be
ordered according to their modules
$$
0\leq |\omega_{1,h}| \leq |\omega_{2,h}| \leq |\omega_{3,h}| \leq \cdots
$$
and again repeated according to their multiplicities.
The convergence of the eigenvalues and absence of spurious modes can, for
instance, be stated as follows (see~\cite[Thm.\ 9.1]{acta_num}).
\begin{theorem}\label{thm:compact_set}
Let us assume that the convergence in norm \eqref{eq:convergence_of_T} is
satisfied. For all compact sets $K$ in the complex plane that do not contain
any eigenvalue of the continuous problem, there exists $h_0$ such that for all
$h<h_0$ no eigenvalue of the discrete problem belongs to $K$. 
\end{theorem}

As a consequence of the previous theorem, a more concrete representation of
the eigenvalue convergence in the complex plane can be given as follows.
\begin{property}\label{def:convergence_of_eigenvalues}
Let $B$ be a ball with radius $R >0$ and let $S = \{\omega_i\}_{i=1}^{n}$ be
the set of $n$ (depending on $R$) eigenvalues, counted with their
multiplicity, being inside $B$, (i.e. $|\omega_i|< R$).
Then, $\forall R >0$ and $\forall \epsilon > 0$, $ \exists h_0$ such that for
$h<h_0$ exactly $n$ discrete eigenvalues, counted with their multiplicity, are
inside $B$, that is $|\omega_{i,h}| < R$. Moreover, the $n$ discrete
eigenvalues can be sorted such that
\begin{equation}
|\omega_i - \omega_{i,h}| <  \epsilon \quad i= 1,2 \dots, n.
\end{equation}
\end{property}

In~\cite{FB_LE} the following theorem was stated concerning the rate of
convergence of the eigenvalues. We are now going to prove it rigorously by
using the abstract setting that we have introduced.
\begin{theorem}\label{thm:rate_of_converg}
Let $\omega_i = \omega_{i+1}= \dots = \omega_{i+m-1}$ be an eigenvalue
of~\eqref{eq:variationa_form_eig} of multiplicity $m$ and denote by $E =
\bigoplus_{j=i}^{i+m-1} E_j$ the corresponding eigenspace. Let $E^*$ be the
corresponding eigenspace of the adjoint problem.
If the discrete spaces $\Sigma_h$ and $U_h$ satisfy the approximation properties 
in Theorem~\ref{thm:uniform_convergence}, then for h small enough 
(so that $\dim U_h \geq i+m-1$) the m discrete eigenvalues 
$\omega_{i,h},\ \omega_{i+1,h},\dots,\ \omega_{i+m-1,h}$
of~\eqref{eq:disdrete_variationa_form_eig} converge to $\omega_i$.
Denote by $E_h$ the direct sum of the discrete eigenspaces and consider the
following quantities
\[
\rho(h)=\sup_{\substack{\normalfont\textbf{u}\in E \\ ||\normalfont\textbf{u}||=1}}
\inf_{\substack{{\underline{\bm{\tau}}\in \Sigma_h} \\
\textbf{v} \in U_h}}
(||\underline{\bm{\sigma}}-\underline{\bm{\tau}}||_{\mathbf{div}}+ 
||\textbf{u}-\textbf{v}||_1)
\]
\[
\rho^*(h)=\sup_{\substack{\normalfont\textbf{p}\in E^*\\ ||\normalfont\textbf{p}||=1}}
\inf_{\substack{{\underline{\bm{\xi}}\in \Sigma_h} \\
\textbf{q} \in U_h}}
(||\underline{\bm{\chi}}-\underline{\bm{\xi}}||_{\mathbf{div}}+ 
||\textbf{p}-\textbf{q}||_1)
\]
where $\underline{\bm{\sigma}}$ is the other component of the solution
of~\eqref{eq:variationa_form_eig} corresponding to $\textbf{u}$ and
analogously $\underline{\bm{\xi}}$ is the other component corresponding to
$\textbf{p}$ for the dual problem.

Then the following error estimates hold true
$$\delta(E,E_h)\leq C\rho(h),$$
$$|\omega_j-\omega_{j,h}| \leq C\rho(h)\rho^*(h) \quad j = i,\dots, i+m-1,$$
where the gap $\delta$ is evaluated in the $H^1(\Omega)$ norm.
\end{theorem}
\begin{proof}

The estimate for the eigenspaces is a direct consequence of
Theorem~\ref{th:eigfun} and of the a priori estimates for the approximation of
the source problem~\eqref{eq:variationa_form_source}.
Indeed, since $T$ is the solution operator associated with $\textbf{u}$ (with $T
\textbf{f} = \Tlam \textbf{f}=\textbf{u}$), by the definition of the
operator norm and the standard energy estimates, and by considering the other
component $\underline{\bm{\sigma}}$ of the solution of~\eqref{eq:variationa_form_source},
we have
\begin{align*}
\delta(E,E_h)
&\leq C ||(T-T_h)_{|E}||_{\mathcal{L}(H^1)} \\
& = C \sup_{\substack{{\textbf{f}} \in E \\||\textbf{f}||_1=1 }} 
||(T-T_h)\textbf{f}||_1  \\
& \le C \sup_{\substack{{\textbf{u}} \in E \\||\textbf{u}||_1=1 }} 
||\textbf{u}-\textbf{u}_h||_1 \\
& \leq C \sup_{\substack{{\textbf{u}} \in E \\||\textbf{u}||_1=1 }}
 (||\textbf{u}-\textbf{u}_h||_1 
+ 
||\underline{\bm{\sigma}} -\underline{\bm{\sigma}}_h||_{\textbf{div}}) \\
& \leq C \sup_{\substack{{\textbf{u}} \in E \\||\textbf{u}||_1=1 }} 
\inf_{\substack{{\underline{\bm{\tau}}\in \Sigma_h} \\
\textbf{v} \in U_h}}
(||\textbf{u}-\textbf{v}||_1 
+ 
||\underline{\bm{\sigma}} - \underline{\bm{\tau}} ||_{\textbf{div}}) \\
& = C \rho(h).
\end{align*}

In order to show the estimate for the eigenvalues we are going to use the
results presented in Theorem~\ref{thm:bound_eigvals_with_T_T*}. The
correspondence between the abstract setting and our problem is given by the
following identifications
\[
\aligned
&\gX=\underline{\textbf{X}}_N\times H^1_{0,D}(\Omega)^d\\
&\gU=(\underline{\bm{\sigma}},\textbf{u})\\
&\gV=(\underline{\bm{\tau}},\textbf{v})\\
&\gA(\gU,\gV)=
(\mathbf{\mathcal{A}}\underline{\bm{\sigma}},\mathbf{\mathcal{A}}\underline{\bm{\tau}})
+
(\textbf{div}\underline{\bm{\sigma}},\textbf{div}\underline{\bm{\tau}})
-
(\mathbf{\mathcal{A}}\underline{\bm{\tau}},\underline{\bm{\epsilon}}(\textbf{u}))
-
(\mathbf{\mathcal{A}}\underline{\bm{\sigma}}, \underline{\bm{\epsilon}}(\textbf{v}))
+
(\underline{\bm{\epsilon}}(\textbf{u}),\underline{\bm{\epsilon}}(\textbf{v}))\\
&\gB(\gU,\gV)=-(\textbf{u},\textbf{div}\underline{\bm{\tau}}).
\endaligned
\]

With these identifications we are now in a position to give a representation
of the solution operator $\gT$ and of its adjoint $\gT^*$. It is interesting
to remark that the bilinear form $\gA$ is symmetric, so that the non symmetry
of the problem comes from the fact that the bilinear form on the right hand
side $\gB$ is not symmetric. For this reason, according to~\eqref{eq:sym}, the
adjoint operator will be constructed by keeping the same left hand side of our
equation and by considering the transpose of the right hand side.

More precisely, given $\gF=\gG=(\underline{\textbf{G}},\textbf{f})\in\gX$,
$\gT\gF$ will be given by $(\underline{\bm{\sigma}},\textbf{u})\in\gX$
solution of~\eqref{eq:variationa_form_source}, while $\gT^*\gG$ is given by
the solution $(\underline{\bm{\chi}},\textbf{p})\in\gX$ of the following dual
problem

\begin{equation}\label{eq:dual_problem_adjoint}
\begin{cases}
(\mathbf{\mathcal{A}}\underline{\bm{\chi}},\mathbf{\mathcal{A}}\underline{\bm{\xi}})
+
(\textbf{div}\underline{\bm{\chi}},\textbf{div}\underline{\bm{\xi}})
-
(\mathbf{\mathcal{A}}\underline{\bm{\xi}},\underline{\bm{\epsilon}}(\textbf{p}))
= 
\bm{0}
&\quad \forall \underline{\bm{\xi}} \in \underline{\bm{X}}_N, \\
-
(\mathbf{\mathcal{A}}\underline{\bm{\chi}}, \underline{\bm{\epsilon}}(\textbf{v}))
+
(\underline{\bm{\epsilon}}(\textbf{p}), 
\underline{\bm{\epsilon}}(\textbf{v}))
=
-(\textbf{div}\underline{\textbf{G}}, \textbf{v})
&\quad \forall\textbf{v} \in H^{1}_{0,D}(\Omega)^d.
\end{cases}
\end{equation}

The definition of the discrete operators $\gT_h$ and $\gT^*_h$ is completely
analogous and makes use of the discrete space $\gX_h=(\Sigma_h,U_h)$.

Theorem~\ref{thm:bound_eigvals_with_T_T*} bounds the error in the eigenvalues
with the sum of two terms.
The bound for the first term uses the same arguments as in the proof of
Theorem~\ref{thm:uniform_convergence}, together with the definitions of
$\rho(h)$ and $\rho^*(h)$. Let us focus on the second term.

We have to estimate the product of the following two quantities
\[
\aligned
&\|(\gT-\gT_h)_{|\gE}\|_{\mathcal{L}(\gX)}\\
&\|(\gT^*-\gT^*_h)_{|\gE^*}\|_{\mathcal{L}(\gX)}
\endaligned
\]
Actually, the result will follow by observing that the first term is bounded by
$C\rho(h)$ and the second one by $C\rho^*(h)$.

If $(\underline{\bm{\sigma}},\textbf{u})$ is an eigenfunction in $\gE$, then
the following a priori estimate is valid
\[
\|\underline{\bm{\sigma}}-\underline{\bm{\sigma_h}}\|_{\textbf{div}}+
\|\textbf{u}-\textbf{u}_h\|_1\le C\rho(h)\|\textbf{u}\|_0
\]
Since $\textbf{u}$ is an eigenfunction, its $L^2(\Omega)$-norm is bounded by
the $H^1(\Omega)$-norm and we have that $\|\textbf{u}\|_1\le\|\gU\|_{\gX}$, so
that we can conclude the desired result
\[
\|(\gT-\gT_h)_{|\gE}\|_{\mathcal{L}(\gX)}\le C\rho(h).
\]

The estimate for the adjoint operator is performed analogously observing that,
if $(\underline{\bm{\chi}},\textbf{p})$ is an eigenfunction in $\gE^*$, then the
following a priori bound holds true
\[
\|\underline{\bm{\chi}}-\underline{\bm{\chi_h}}\|_{\textbf{div}}+
\|\textbf{p}-\textbf{p}_h\|_1\le C\rho^*(h)\|\textbf{div}\underline{\bm{\chi}}\|_0
\]
By the standard relations between the eigenspaces $\gE$ and $\gE^*$ and the
fact that $\underline{\bm{\chi}}$ is an eigenfunction of the dual problem we
can finally conclude that
\[
\|(\gT^*-\gT^*_h)_{|\gE^*}\|_{\mathcal{L}(\gX)}\le C\rho^*(h).
\]

\end{proof}

\begin{remark}
A natural question is whether the results of the previous theorem guarantee
the double order of convergence for the approximation of the eigenvalues. The
double order of convergence would follow if we can show that $\rho^*(h)$ is
asymptotically equivalent to $\rho(h)$.
This is a tricky question because it involves the regularity of the dual
problem.  In general, as already mentioned in Remark~\ref{re:duality}, we
cannot expect from the solution of the dual problem the same regularity as for
the original one. On the other hand, we do not need the regularity of a
generic solution, but only of the solution corresponding to an eigenfunction.
At the moment, only partial results in this direction are available; in
particular, we can prove that the dual solution of the FOSLS approximation of
the Poisson equation has the same regularity as the original one when an
eigenfunction is considered~\cite{FB_eig}. The same result can be proved for
some DPG formulations~\cite{DPG}. On the other hand, the numerical results of
the next section confirm that in our examples the double order of convergence
is achieved.
\end{remark}

\section{Numerical Results}
\label{se:num}

In this section we present several numerical results related to the
formulation that we have described.
Some preliminary results were already present in~\cite{FB_LE}. However in that
case, for simplicity, a special constitutive law was considered so that the
problem was equivalent to a Stokes system.
Now we solve the true linear elasticity problem and we are interested in showing
how the reported numerical results are robust with respect to the Lam\'e
constants, in particular when tending to the incompressible limit. The
behavior of the computed spectrum will give us useful information on some
properties of the solution operator.

We are not focusing on the efficiency of the numerical solver, but we are
mainly interested in representing the computed spectrum in the complex plane
as a function of the mesh and of the Lam\'e constants.

To this end, we consider a $2\text{D}$ linear elasticity problem with 
homogeneous Dirichlet boundary conditions, where the compliance tensor is given by
\begin{equation}
\mathcal{A}\underline{\bm{\tau}} =
\frac{1}{2\mu}(\underline{\bm{\tau}} -
\frac{\lambda}{2\lambda + 
2\mu}\text{tr}(\underline{\bm\tau})\underline{\mathbf{I}}).	
\end{equation}
In all our tests the Lam\'e constants are $\mu=1$ and $\lambda$ varying from
$1$, $100$, $10^4$ to $ 10^8$.
Moreover, different mesh structures are studied and investigated. 

We start by presenting the rate of convergence for the first and second eigenvalues on different meshes.
Then, the spread of eigenvalues in the complex plane is explored for each mesh as it is refined.

In order to avoid artifacts introduced by the fact that we are rewriting the
elasticity equation as a first order system, we considered only \emph{genuine}
eigenvalues of~\eqref{eq:variationa_form_eig} in the sense that they
correspond to eigenfunctions for which the displacement is not zero. These
eigenvalues are related to the Schur complement described in~\eqref{eq:Shur_u}.

A first order scheme based on Raviart--Thomas elements $(RT_0)$ for the 
approximated space $\Sigma_h$ is considered with continuous Galerkin of degree one for the space $U_h$, 
with $h$ being the mesh size or step size.

These finite element spaces satisfy the conditions of Theorem \ref{thm:uniform_convergence}.
The general matrices in System~\ref{eq:discte_matrix_structure} were
constructed with FEniCS~\cite{fenics} and then matrices corresponding to the
Schur complement in~\ref{eq:Schur_complement} were extracted and used to solve the eigenvalue problem.
\subsection{Rate of convergence on the square domain} 
\label{sub:rate_of_convergence_square}
In order to confirm the convergence rate stated in Section \ref{se:convergence
analysis}, we first consider a square domain $\Omega = ]0,1[^2$ and three different
mesh sequences.
\begin{figure}[t!]
\centering
\begin{subfigure}[b]{0.3\textwidth}
\centering
\includegraphics[width=1.0\textwidth]{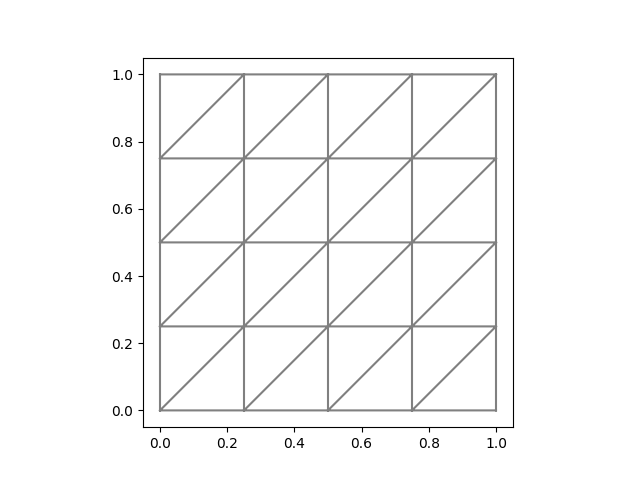}
\caption{Right}
\label{fig:right_square}
\end{subfigure}
\hfill
\begin{subfigure}[b]{0.3\textwidth}
\centering
\includegraphics[width=1.0\textwidth]{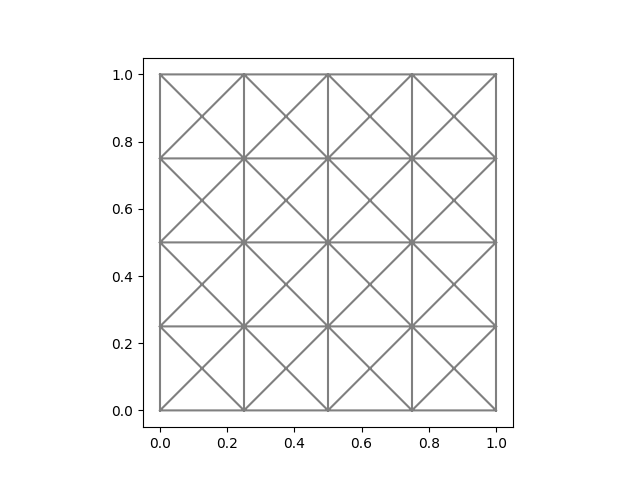}
\caption{Crossed}
\label{fig:crossed_square}
\end{subfigure}
\hfill
\begin{subfigure}[b]{0.3\textwidth}
\centering
\includegraphics[width=1.0\textwidth]{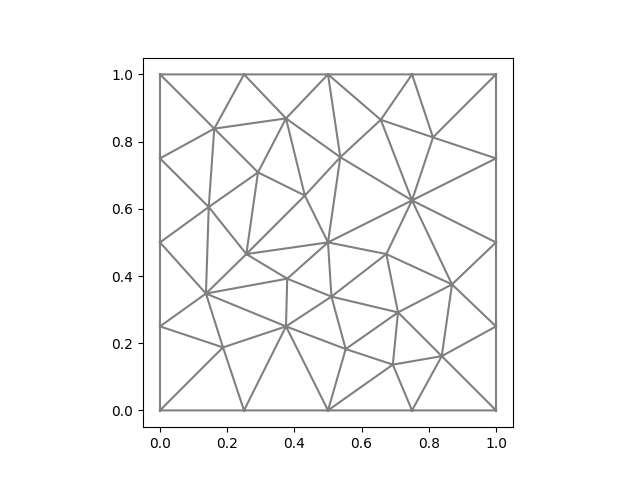}
\caption{Nonuniform}
\label{fig:nonuniform_square}
\end{subfigure}
\caption{Meshes on the unit square with N = 4}
\label{fig:square_meshes}
\end{figure}
We use the following meshes: a non-symmetric uniform mesh (Right), a symmetric
uniform mesh (Crossed), and a non-symmetric and non-structured mesh
(Nonuniform). 
The integer $N$ refers to the number of subdivisions on each side of the square. 
An example of such meshes is given in Figure \ref{fig:square_meshes}.
\begin{table}[t!]
\centering
\begin{tabular}{ c|c|c|c|c }

&  $\lambda = 1$ 		& $\lambda = 100$   & $\lambda = 10^4$ & $\lambda = 10^8$ \\ \hline
$\omega_1$  & 37.266072200953786    & 52.31315105053875 & 52.3443693  	   & 52.344691 \\ \hline
$\omega_2$  & 37.2660721997643 		& 91.4778227239564  &92.11827609964527 & 92.12439336305897 \\
\end{tabular}
\caption{Estimates of the first and second eigenvalues on a square} 
\label{tab:estimates_of_eigenvalues_on_square}
\end{table}
We computed with FEniCS the reference solutions providing an estimate of the first and
second eigenvalues for the different cases we want to approximate. We used a
high order scheme and an adaptive algorithm with the classical
\uppercase{Solve-Estimate-Mark-Refine} strategy.
The values of the reference eigenvalues are reported in Table \ref{tab:estimates_of_eigenvalues_on_square}.
\begin{figure}[t!]
\centering
\begin{subfigure}[b]{0.3\textwidth}
\centering
\includegraphics[width=\textwidth]{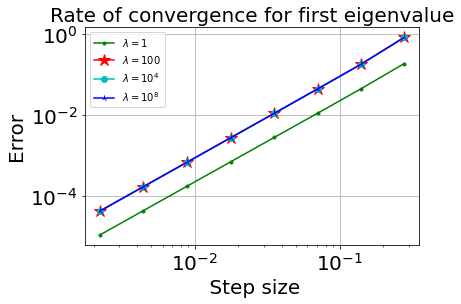}
\caption{Right mesh}
\end{subfigure}
\hfill
\begin{subfigure}[b]{0.3\textwidth}
\centering
\includegraphics[width=\textwidth]{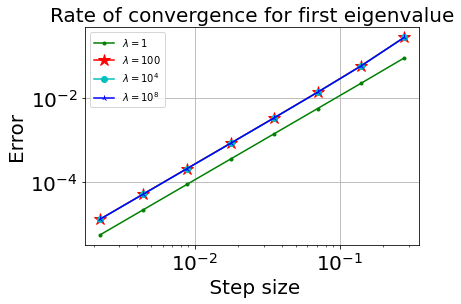}
\caption{Crossed mesh}
\end{subfigure}
\hfill
\begin{subfigure}[b]{0.3\textwidth}
\centering
\includegraphics[width=\textwidth]{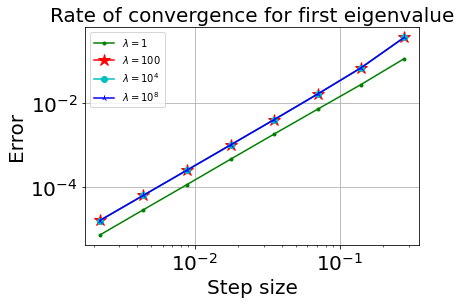}
\caption{Nonuniform mesh}
\end{subfigure}
\vfill
\centering
\begin{subfigure}[b]{0.3\textwidth}
\centering
\includegraphics[width=\textwidth]{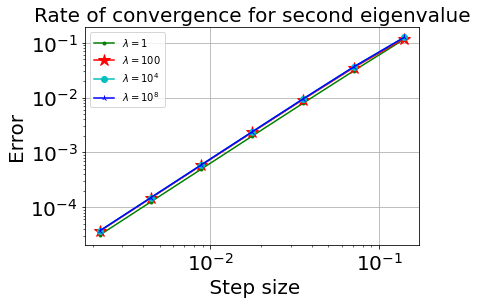}
\caption{Right mesh}
\end{subfigure}
\hfill
\begin{subfigure}[b]{0.3\textwidth}
\centering
\includegraphics[width=\textwidth]{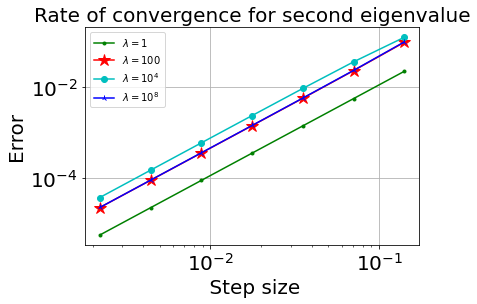}
\caption{Crossed mesh}
\end{subfigure}
\hfill
\begin{subfigure}[b]{0.3\textwidth}
\centering
\includegraphics[width=\textwidth]{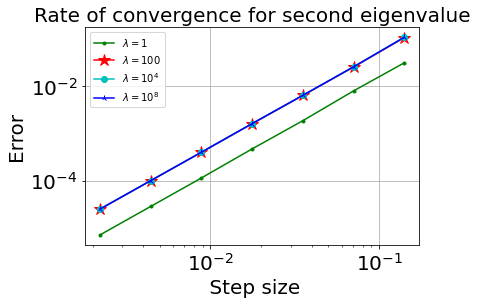}
\caption{Nonuniform mesh}
\end{subfigure}
\caption{Rate of convergence for the first and second eigenvalues on squared domain}
\label{fig:rate_of_conv_first_eig_square}
\end{figure}\\
Figure \ref{fig:rate_of_conv_first_eig_square} shows the rate of convergence 
for both the first and the second approximated eigenvalues.
Clearly both eigenvalues match the theoretical result in Theorem \ref{thm:rate_of_converg} with the rate being of second order on all meshes provided. 
\subsection{Rate of convergence on the L-shaped domain} 
\label{sub:rate_of_convergence_lshape}
The second domain that is considered is the L-shaped domain obtained by
removing a quarter of the square $]0,1[^2$.
It is known that on such domain singularities due to the re-entrant corner may
occur.
We also consider three sequences of triangulations: (Left), (Crossed or
Uniform), and (Nonuniform) meshes.
These are illustrated in Figure \ref{fig:lshape_mesh}  with $N = 4$.
\begin{figure}[t!]
\centering
\begin{subfigure}[b]{0.3\textwidth}
\centering
\includegraphics[width=\textwidth]{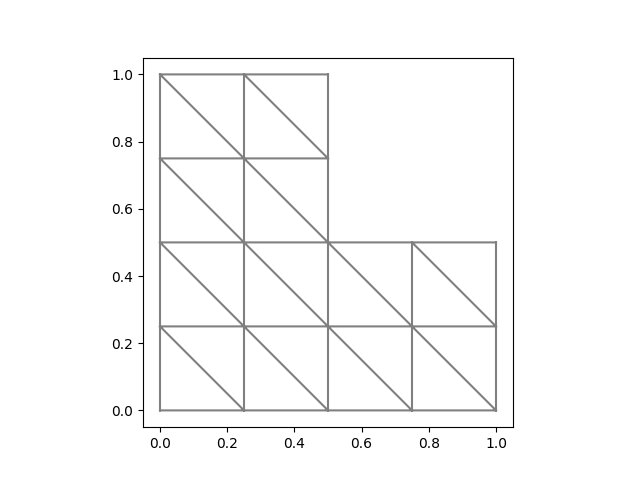}
\caption{Left}
\label{fig:lshape_left}
\end{subfigure}
\hfill
\centering
\begin{subfigure}[b]{0.3\textwidth}
\centering
\includegraphics[width=\textwidth]{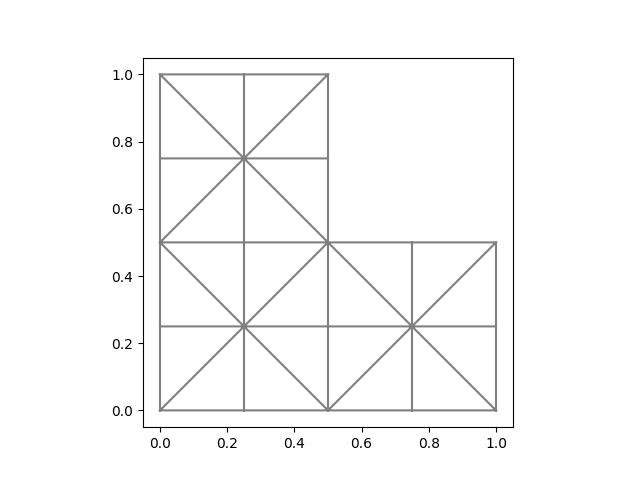}
\caption{Uniform}
\label{fig:lshaped_uniform}
\end{subfigure}
\hfill
\begin{subfigure}[b]{0.3\textwidth}
\centering
\includegraphics[width=\textwidth]{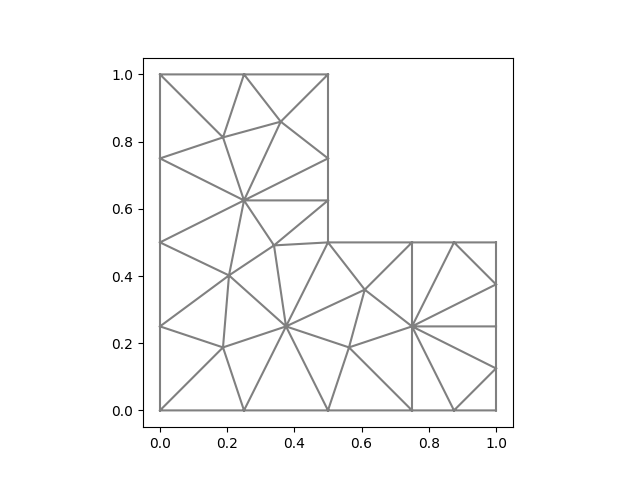}
\caption{Nonuniform}
\label{fig:lshaped_Non_uniform}
\end{subfigure}
\caption{Meshes for L-shaped domain with $N=4$} 
\label{fig:lshape_mesh}
\end{figure}
In order to compare the approximated eigenvalues with a reference value, the
same adaptive code, as in the previous case, was used to find the estimate for the first and second 
eigenvalues for different $\lambda$.
\begin{table}[t!]
\centering
\begin{tabular}{ c|c|c|c|c }

&  $\lambda = 1$ 		& $\lambda = 100$     & $\lambda = 10^4$   & $\lambda = 10^8$ \\ \hline
$\omega_1$  & 54.36578831544661     & 127.990463	      & 128.52363885700083 & 128.5293816640767 \\ \hline
$\omega_2$  & 69.08352886532845 	& 147.44431322194393  &148.06735898422232  & 148.07344440728022 \\
\end{tabular}
\caption{Estimates of the first and second eigenvalues on an L-shaped domain} 
\label{tab:estimates_of_eigenvalues_on_L-shaped}
\end{table}
These estimated values are reported in Table \ref{tab:estimates_of_eigenvalues_on_L-shaped}.
\begin{figure}[ht!]
\begin{subfigure}[b]{0.3\textwidth}
\centering
\includegraphics[width=\textwidth]{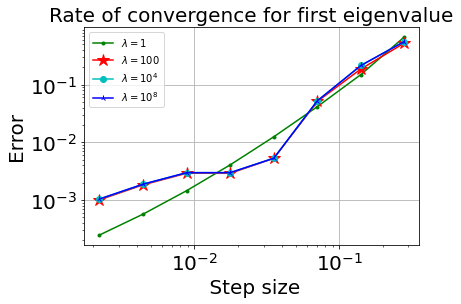}
\caption{Left mesh}
\label{fig:left_lshape_mesh_first_eig_rate}
\end{subfigure}
\hfill
\begin{subfigure}[b]{0.3\textwidth}
\centering
\includegraphics[width=\textwidth]{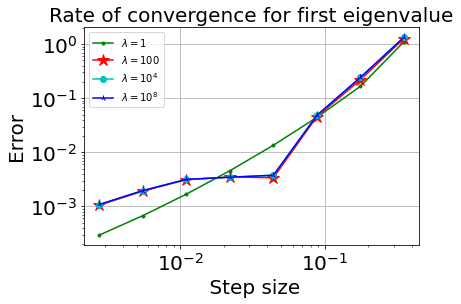}
\caption{Uniform mesh}
\label{fig:uniform_lshape_mesh_first_eig_rate}
\end{subfigure}
\hfill
\begin{subfigure}[b]{0.3\textwidth}
\centering
\includegraphics[width=\textwidth]{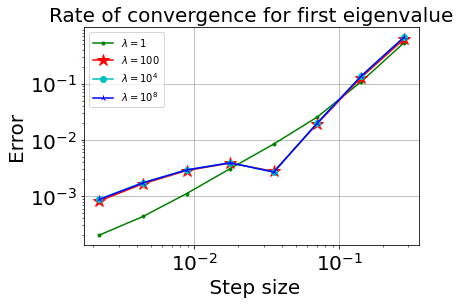}
\caption{Nonuniform mesh}
\label{fig:Nonumiform_lshape_mesh_first_eig_rate}
\end{subfigure}
\vfill
\centering
\begin{subfigure}[b]{0.3\textwidth}
\centering
\includegraphics[width=\textwidth]{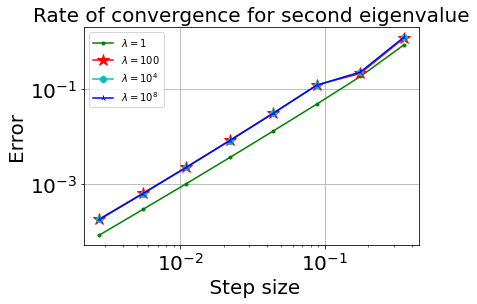}
\caption{Left mesh}
\label{fig:left_lshape_mesh_second_eig_rate}
\end{subfigure}
\hfill
\begin{subfigure}[b]{0.3\textwidth}
\centering
\includegraphics[width=\textwidth]{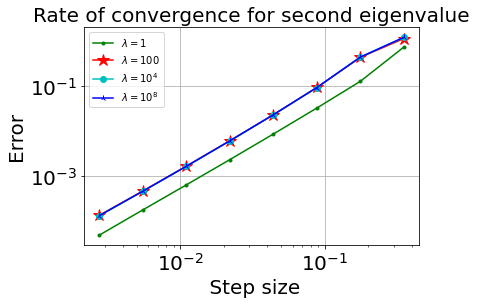}
\caption{Uniform mesh}
\label{fig:Uniform_lshape_mesh_second_eig_rate}
\end{subfigure}
\hfill
\begin{subfigure}[b]{0.3\textwidth}
\centering
\includegraphics[width=\textwidth]{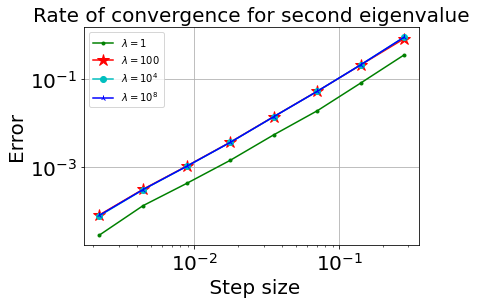}
\caption{Nonuniform mesh}
\label{fig:Nonumiform_lshape_mesh_second_eig_rate}
\end{subfigure}
\caption{Rate of convergence for the first and second eigenvalues L-shaped domain}
\label{fig:rate_of_conv_first_and_second_eig_Lshaped}
\end{figure}
It is well known that the rate of convergence is reduced when the
eigenfunction is singular. This is expected in particular in the case of the
first eigenvalue.

When looking at the computed results, a particular phenomenon is present when
approximating the first eigenvalue close to the incompressible limit.
This is clearly seen in Figures~\ref{fig:left_lshape_mesh_first_eig_rate},
\ref{fig:uniform_lshape_mesh_first_eig_rate},
and~\ref{fig:Nonumiform_lshape_mesh_first_eig_rate} for all meshes. Actually,
the expected rate of convergence is achieved but there is a strange
pre-asymptotic behavior when the value of $\lambda$ increases.
By looking at the computed eigenfunctions it looks like this phenomenon is
related to a sort of locking caused by the location of the degrees of freedom
in the proximity to the re-entered corner.
\begin{figure}[t!]
\begin{subfigure}{0.45\textwidth}
\centering
\includegraphics[width=35mm]{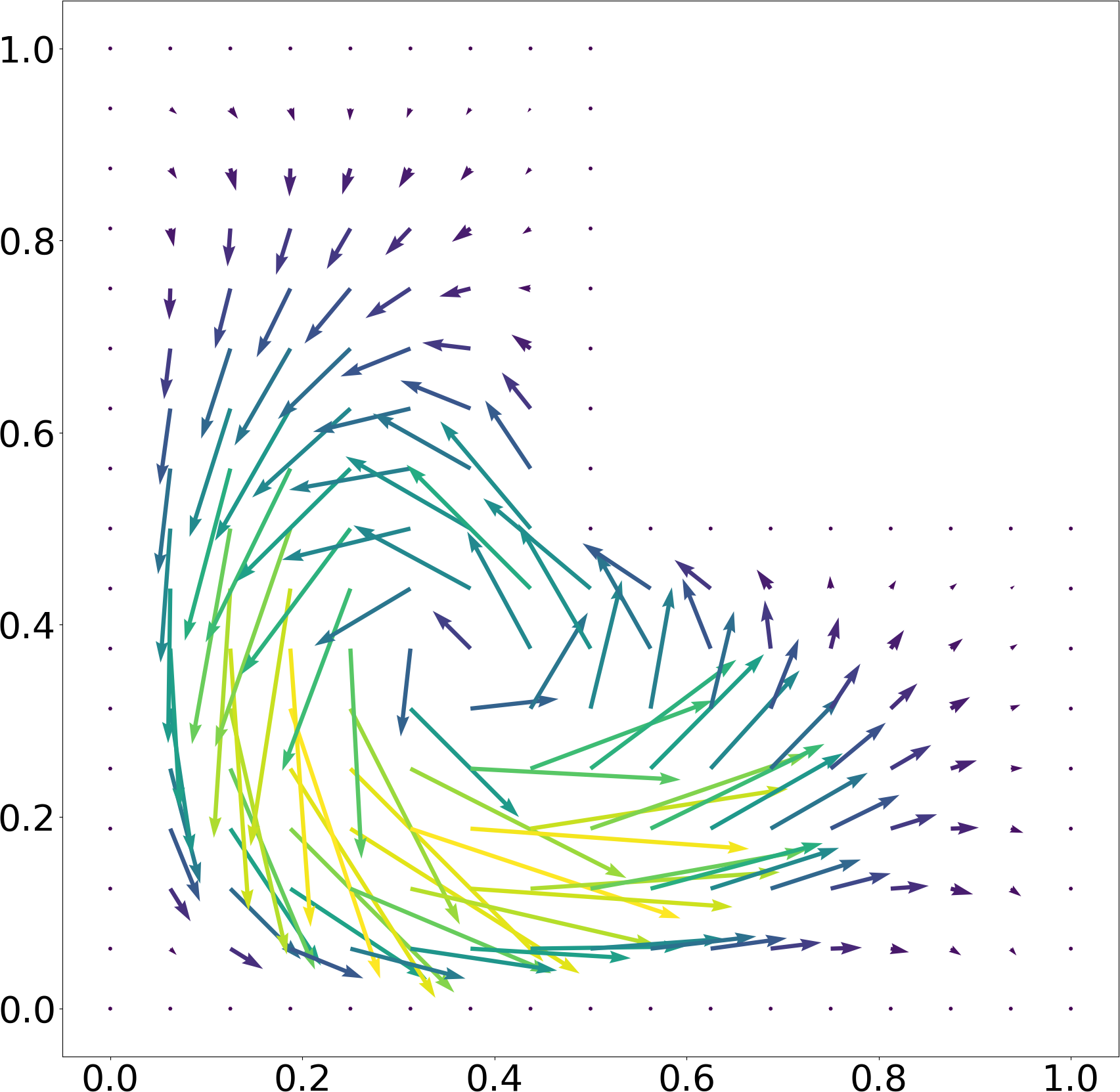}
\caption{Left mesh with $N=16$}
\label{fig:left_lshape_mesh_iteration3}
\end{subfigure}
\hfill
\begin{subfigure}{0.45\textwidth}
\centering
\includegraphics[width=35mm]{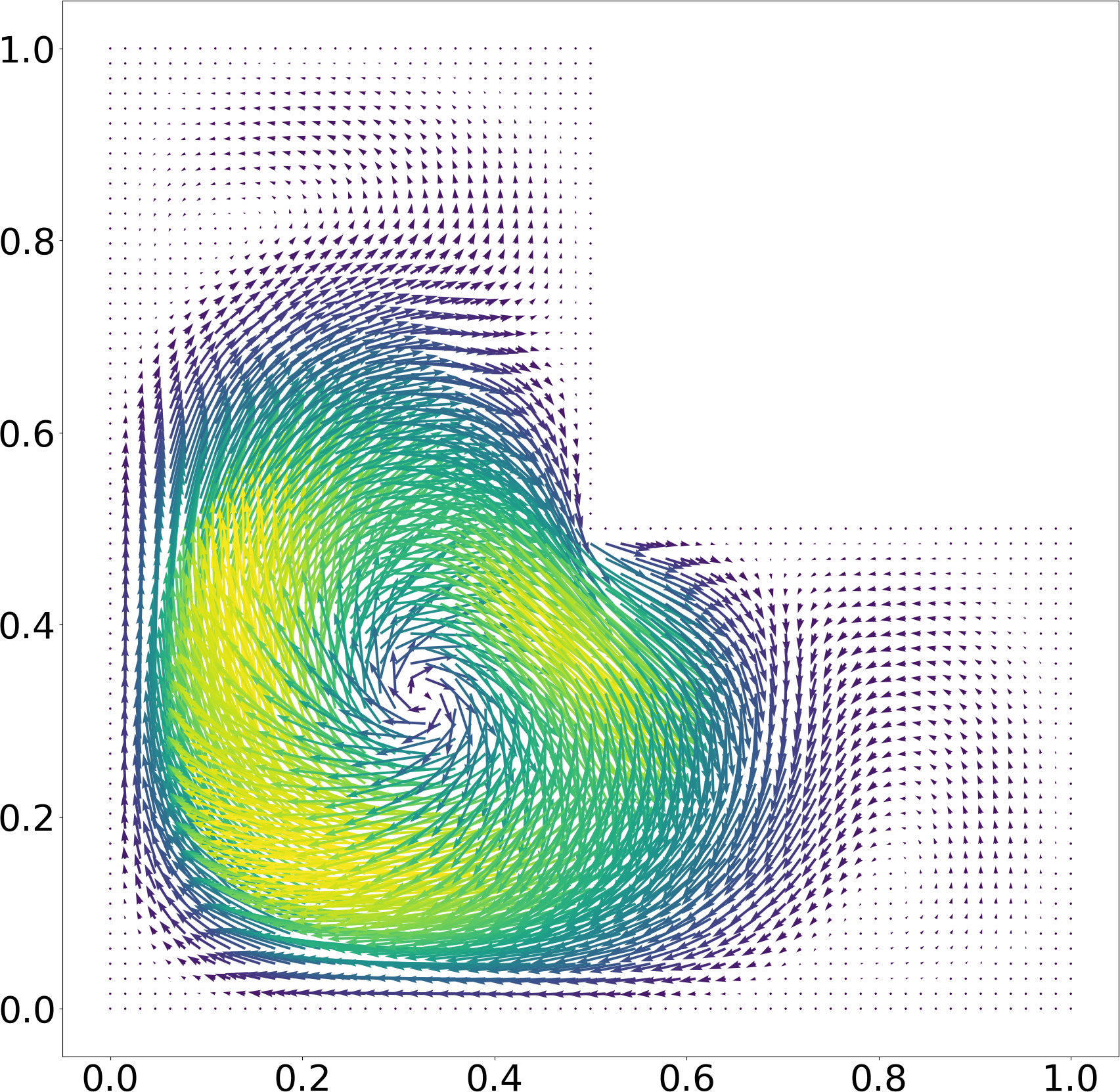}
\caption{Left mesh with $N=64$}
\label{fig:left_lshape_mesh_iteration5}
\end{subfigure}
\vfill
\begin{subfigure}{0.45\textwidth}
\centering
\includegraphics[width=35mm]{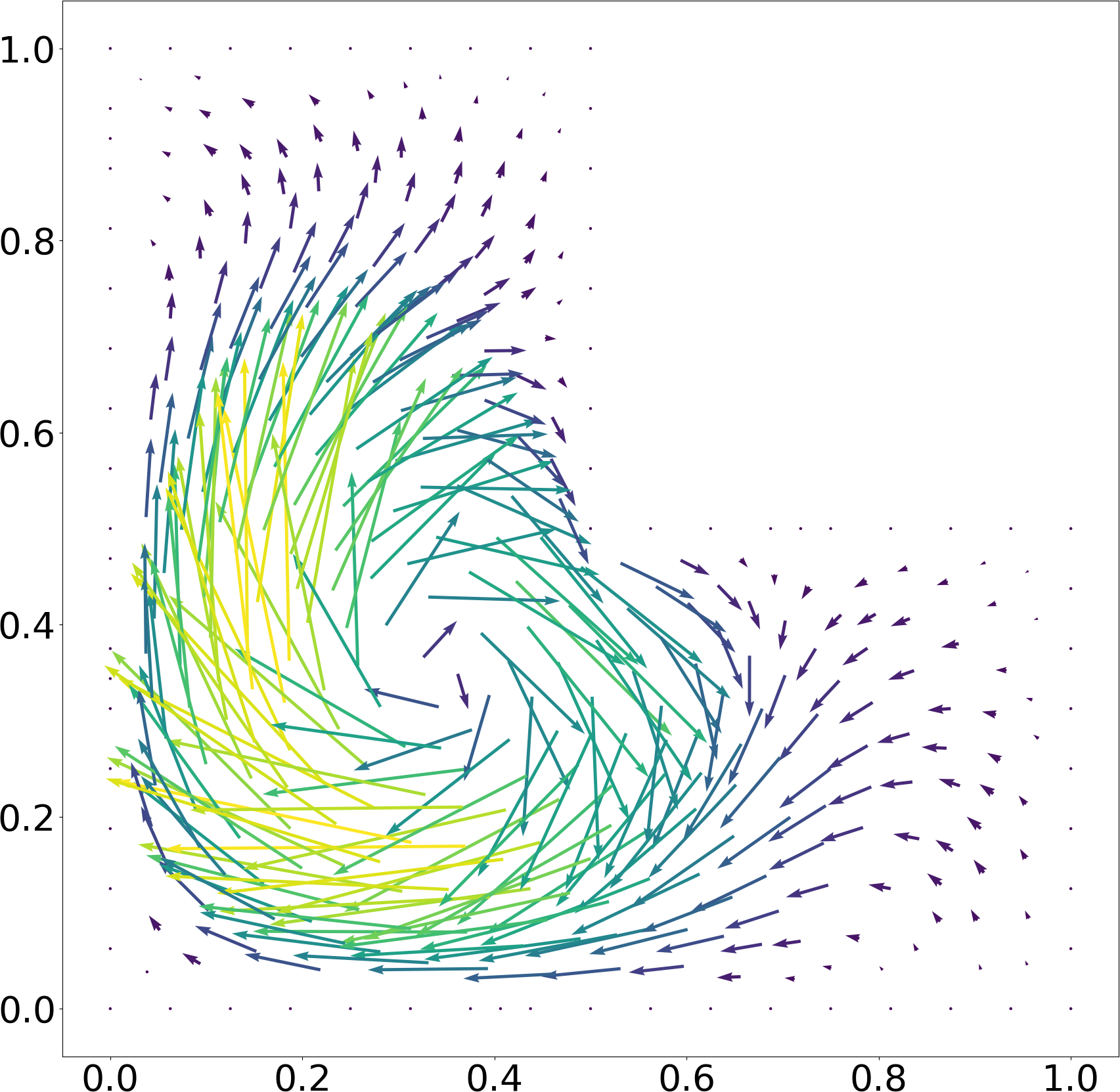}
\caption{Nonuniform mesh with $N=16$}
\label{fig:Nonniform_lshape_mesh_iteration3}
\end{subfigure}
\hfill
\begin{subfigure}{0.45\textwidth}
\centering
\includegraphics[width=35mm]{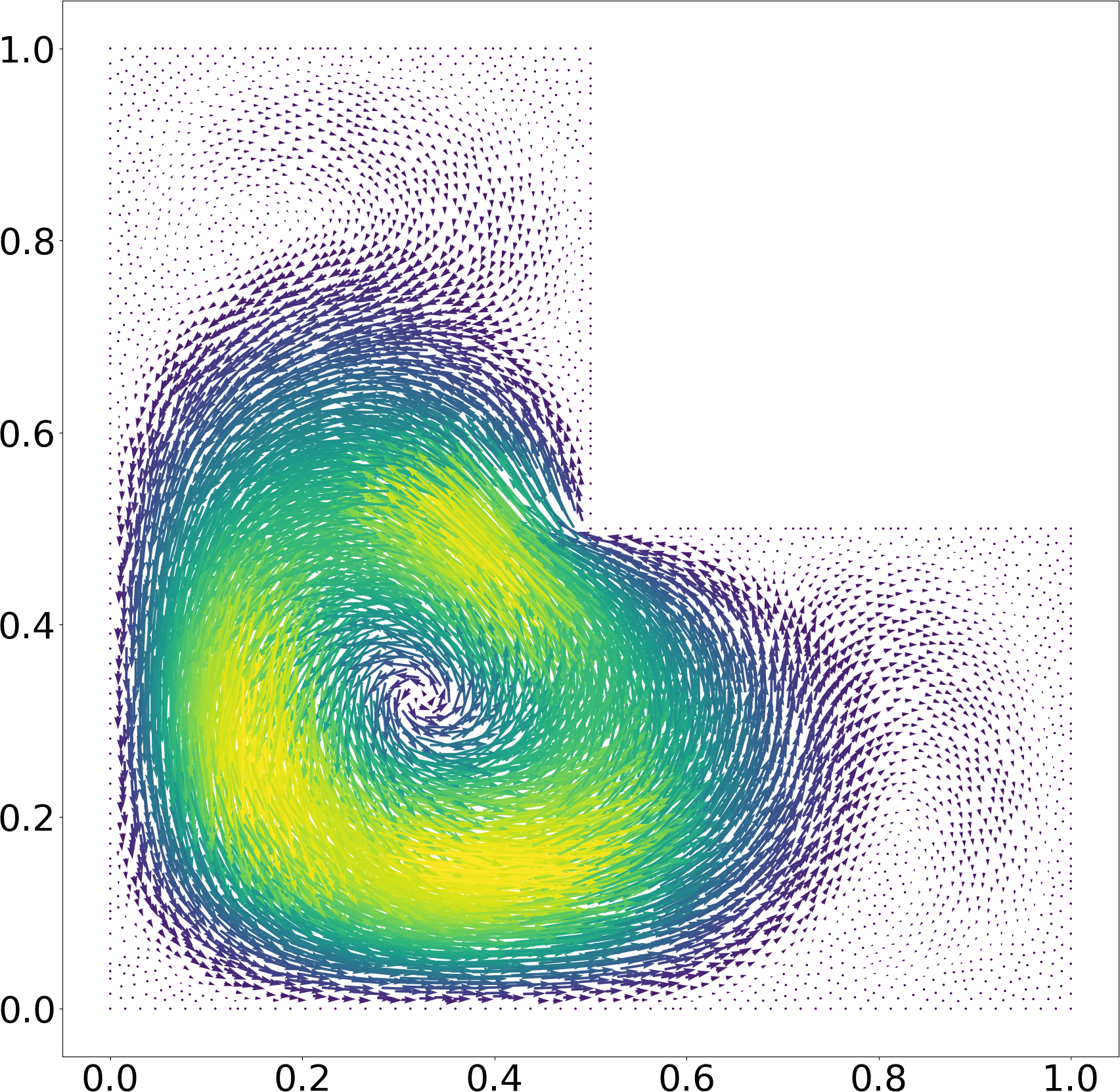}
\caption{Nonuniform mesh with $N=64$}
\label{fig:Nonuniform_lshape_mesh_iteration5}
\end{subfigure}
\caption{First eigenfunction for $\lambda=100$ on L-shaped domain}
\label{fig:first_eigenfunction_Lshaped_left_uniform}
\end{figure}
When the system approaches the incompressible limit, the displacement presents
a recirculation close to the re-entrant corner. A coarse mesh is not capable
to represent such vortex. For $\lambda=1$ this situation doesn't occur.
In order to better describe this phenomenon, we report in
Figures~\ref{fig:left_lshape_mesh_iteration3}
and~\ref{fig:Nonniform_lshape_mesh_iteration3} the eigenfunction corresponding
to the first eigenvalue on the mesh after three refinements for $\lambda=100$.
We see that the correct vortex is not well approximated, possibly because of
the mesh being coarse so that there are not enough degrees of freedom close to
the re-entrant corner.
As we refine, specifically starting from the fifth iteration and onwards,
the expected circulation is noticed as
Figures~\ref{fig:left_lshape_mesh_iteration5}
and~\ref{fig:Nonuniform_lshape_mesh_iteration5} illustrate.
This might be the reason why we need a fine mesh to observe the expected rate
for the first eigenvalue.
\begin{figure}[t!]
\includegraphics[width=35 mm]{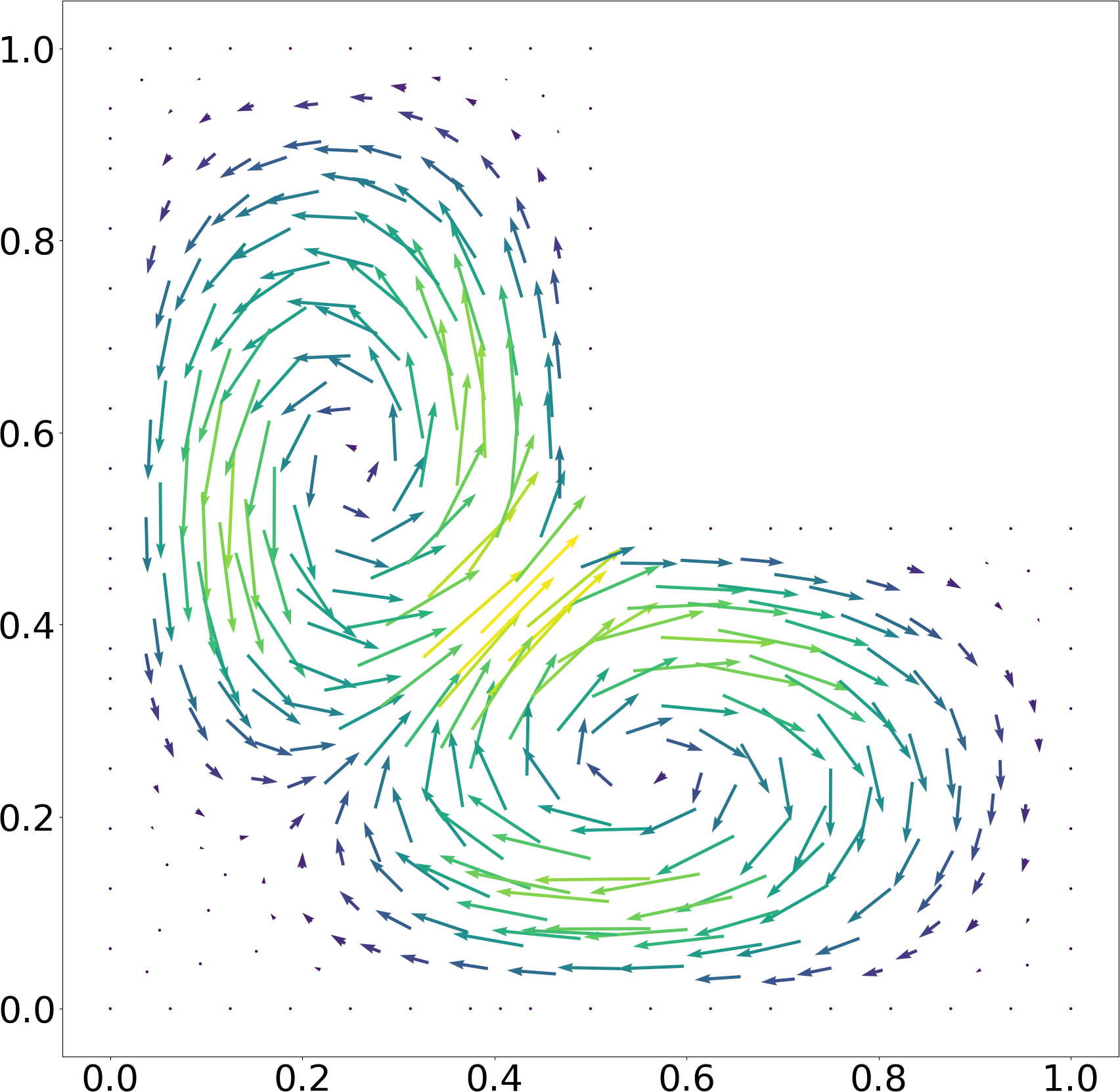}
\caption{Second eigenfunction for $\lambda=100$ on Nonuniform L-shaped domain}
\label{fig:second_eig_Nonuniform_lshape_lam100}
\end{figure}
On the contrary, the re-entrant corner has no negative effects on the second eigenvalue even for coarse meshes. This can be seen, for example, in 
Figure~\ref{fig:second_eig_Nonuniform_lshape_lam100} which illustrates the second eigenfunction on a Nonuniform mesh with $N=16$ and $\lambda=100$.
%
\subsection{The distribution of the eigenvalues in the complex plane} 
\label{sub:the_eigenvalues_spread_in_the_complex_plain}
%
In this section we discuss the distribution of the spectrum of the discrete
operator in the complex plane. This is the main motivation of our paper and
provide us with some properties of the solution operator that are not visible
when considering the source problem. Needless to say, the knowledge of the
spectrum of the operator has important consequences for the design of
numerical schemes, for instance when a transient problem is approximated.

We consider the distribution of the eigenvalues in the case of the two domains
previously considered (square and L-shaped) and the same range of Lam\'e
constants.

We note that we are taking specific outer zoom with the values on the axes being large in
order to show the spread of a large portion of the spectrum.
Moreover, different scales are considered depending on each case in order to
better highlight the behavior of the specific approximation.
\begin{figure}[t!]
\centering
\begin{subfigure}[b]{0.4\textwidth}
\centering
\includegraphics[width=\textwidth]{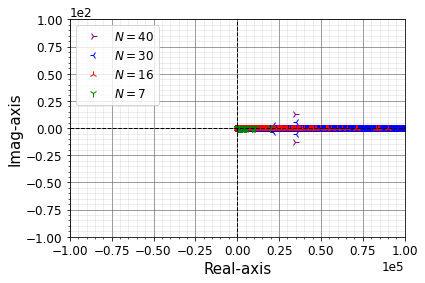}
\caption{$\lambda = 1$}
\label{fig:lam1_right_refine}
\end{subfigure}
\hfill
\begin{subfigure}[b]{0.4\textwidth}
\centering
\includegraphics[width=\textwidth]{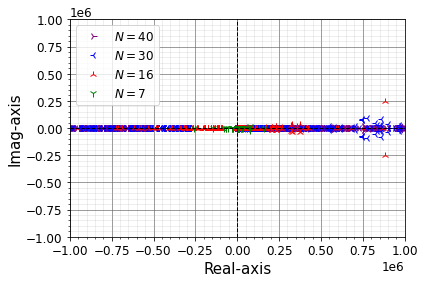}
\caption{$\lambda = 100$}
\label{fig:lam100_right_refine}
\end{subfigure}
\hfill
\begin{subfigure}[b]{0.4\textwidth}
\centering
\includegraphics[width=\textwidth]{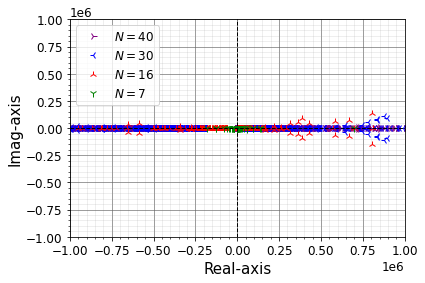}
\caption{$\lambda = 10^4$}
\label{fig:lam10pow4_right_refine}
\end{subfigure}
\hfill
\begin{subfigure}[b]{0.4\textwidth}
\centering
\includegraphics[width=\textwidth]{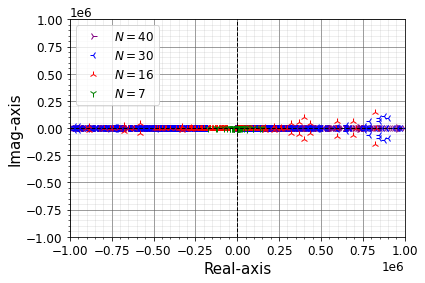}
\caption{$\lambda = 10^8$}
\label{fig:lam10pow8_right_refine}
\end{subfigure}
\caption{Eigenvalues of a square on a refined Right mesh} 
\label{fig:right_refined}
\end{figure}

We recall that the exact eigenvalues are real positive, so that we would
expect that the discrete eigenvalues, for $h$ small enough, distribute along
the positive real axis of the complex plane. This is actually a naive
interpretation of the convergence results presented in the previous sections.
Indeed, by inspecting carefully the result presented in
Property~\ref{def:convergence_of_eigenvalues}, it can be seen that computed
eigenvalues can also be spread in different regions of the complex plane, as
far as they lie outside a ball centered at the origin. The larger the ball,
the smaller in general should $h$ be in order to detect this behavior.

Figure~\ref{fig:right_refined} shows the spectrum of our operator on the Right
mesh for different values of $\lambda$. Each plot reports with different
colors the results computed on four successfully refined meshes.

As shown in Figure \ref{fig:lam1_right_refine}, we observe that for a small 
value of $\lambda$, the spectrum is \emph{well behaving} close to the positive
real axis with some non real eigenvalues appearing as we refine. 
In this case, all values are located in the right half of the complex plane,
that is, they have positive real part.

When $\lambda$ increases, we can see from Figure~\ref{fig:right_refined} that
eigenvalues with negative real part show up. Also in this case some non real
eigenvalues are present, even if their imaginary part is not as large as we
will see in the next examples.
\begin{figure}[t!]
\centering
\begin{subfigure}[b]{0.4\textwidth}
\centering
\includegraphics[width=\textwidth]{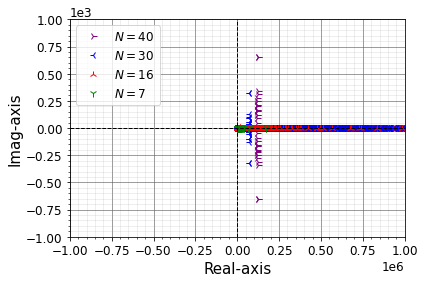}
\caption{$\lambda = 1$}
\label{fig:lam1_crossed_refine}
\end{subfigure}
\hfill
\begin{subfigure}[b]{0.4\textwidth}
\centering
\includegraphics[width=\textwidth]{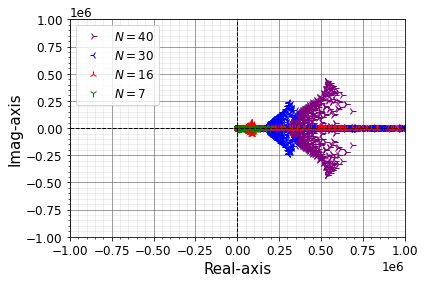}
\caption{$\lambda = 100$}
\label{fig:lam100_crossed_refine}
\end{subfigure}
\hfill
\begin{subfigure}[b]{0.4\textwidth}
\centering
\includegraphics[width=\textwidth]{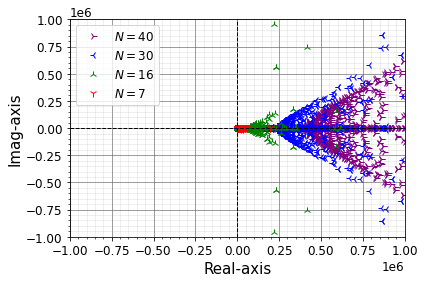}
\caption{$\lambda = 10^4$}
\label{fig:lam10pow4_crossed_refine}
\end{subfigure}
\hfill
\begin{subfigure}[b]{0.4\textwidth}
\centering
\includegraphics[width=\textwidth]{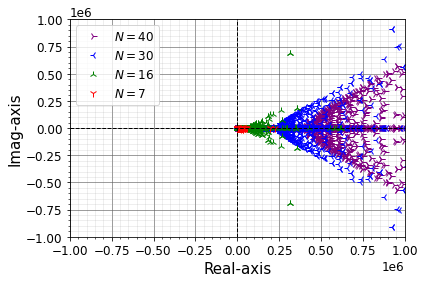}
\caption{$\lambda = 10^8$}
\label{fig:lam10pow8_crossed_refine}
\end{subfigure}
\caption{Eigenvalues of a square on a refined Crossed mesh} 
\label{fig:crossed_refined}
\end{figure}

The results for the Crossed mesh are reported in Figure~\ref{fig:crossed_refined}.
In this case, all eigenvalues have positive real part and are located to the
right side of the complex plane.
The more we refine, the more the scheme produces complex eigenvalues 
moving away from the region of interest and diverging. 
For small values of $\lambda$ (Figure \ref{fig:lam1_crossed_refine}), most eigenvalues are positive and real with some complex being present for finer meshes.
As $\lambda$ becomes larger in Figure \ref{fig:crossed_refined}, the spectrum
is spread more and more with complex eigenvalues having positive real part.
\begin{figure}[t!]
\centering
\begin{subfigure}[b]{0.4\textwidth}
\centering
\includegraphics[width=\textwidth]{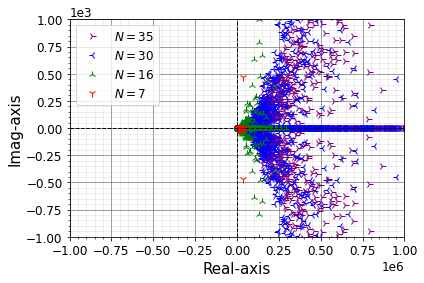}
\caption{$\lambda = 1$}
\label{fig:lam1_unstr_refine}
\end{subfigure}
\hfill
\begin{subfigure}[b]{0.4\textwidth}
\centering
\includegraphics[width=\textwidth]{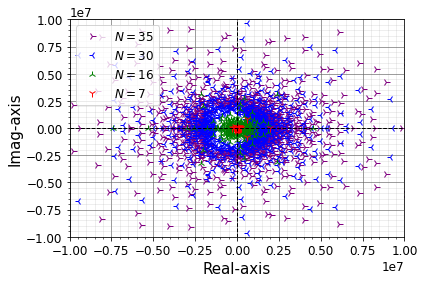}
\caption{$\lambda = 100$}
\label{fig:lam100_unstr_refine}
\end{subfigure}
\hfill
\begin{subfigure}[b]{0.4\textwidth}
\centering
\includegraphics[width=\textwidth]{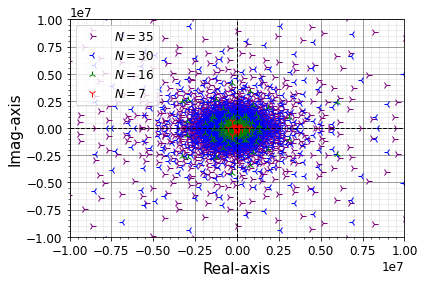}
\caption{$\lambda = 10^4$}
\label{fig:lam10pow4_unstr_refine}
\end{subfigure}
\hfill
\begin{subfigure}[b]{0.4\textwidth}
\centering
\includegraphics[width=\textwidth]{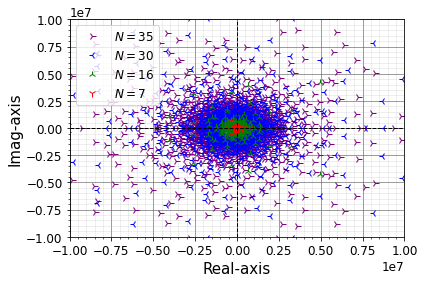}
\caption{$\lambda = 10^8$}
\label{fig:lam10pow8_unstr_refine}
\end{subfigure}
\caption{Eigenvalues of a square on a refined Nonuniform mesh} 
\label{fig:unstr_refined}
\end{figure}

Finally, the spectrum computed with the Nonuniform mesh has a different
structure. The eigenvalues appear to be more spread all over the complex 
plane as Figure \ref{fig:unstr_refined} shows.
As in the previous observations, the values of the eigenvalues, for
$\lambda=1$, are concentrated to the right of the complex plane.
However, in this case, the eigenvalues are scattered much more when comparing
them to the Right and Crossed meshes for $\lambda=1$.
The spectrum gets surprisingly more symmetric with respect to the origin when
moving towards the incompressible limit.

\begin{figure}[t!]
\centering
\begin{subfigure}[b]{0.3\textwidth}
\centering
\includegraphics[width=\textwidth]{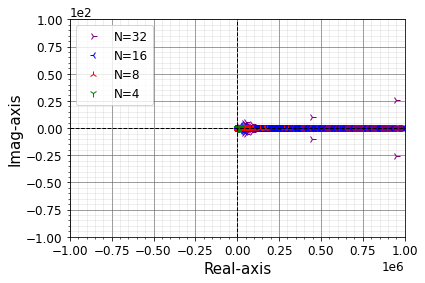}
\caption{$\lambda = 1$}
\end{subfigure}
\hfill
\begin{subfigure}[b]{0.3\textwidth}
\centering
\includegraphics[width=\textwidth]{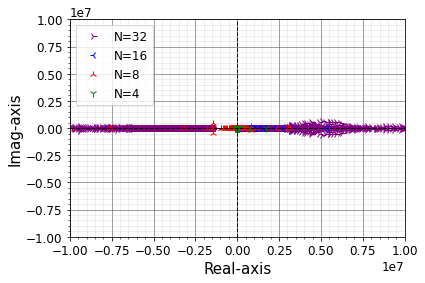}
\caption{$\lambda = 100$}
\end{subfigure}
\hfill
\begin{subfigure}[b]{0.3\textwidth}
\centering
\includegraphics[width=\textwidth]{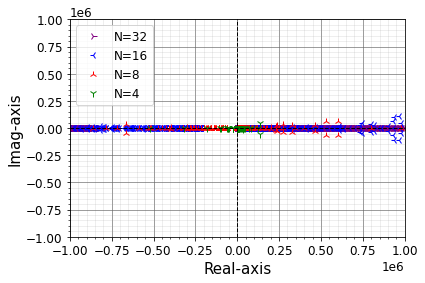}
\caption{$\lambda = 10^8$}
\end{subfigure}
\caption{Eigenvalues of a Left L-shaped domain} 
\label{fig:left_lshaped_different_lam}
\end{figure}

In what follows we present some results for the L-shaped domain. We only present the cases where $\lambda=10^r$ for ($r = 0,2,8$) since there are no significant differences between $r=4$ and $8$.

We start by exploring the Left mesh structure in Figure \ref{fig:left_lshaped_different_lam}.
Looking at the spectrum in this case, we see a similar behavior as for the
case of the Right mesh on the square.
\begin{figure}[t!]
\centering
\begin{subfigure}[b]{0.3\textwidth}
\centering
\includegraphics[width=\textwidth]{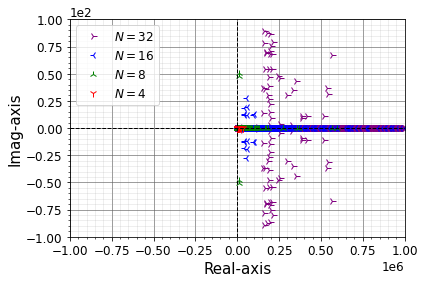}
\caption{$\lambda = 1$}
\end{subfigure}
\hfill
\begin{subfigure}[b]{0.3\textwidth}
\centering
\includegraphics[width=\textwidth]{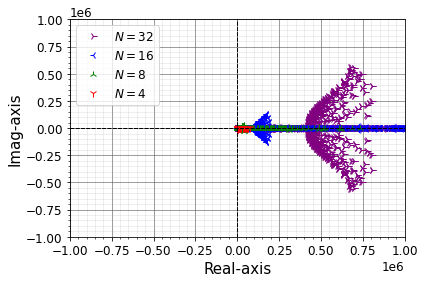}
\caption{$\lambda = 100$}
\end{subfigure}
\hfill
\hfill
\begin{subfigure}[b]{0.3\textwidth}
\centering
\includegraphics[width=\textwidth]{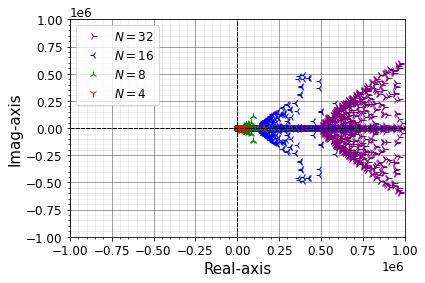}
\caption{$\lambda = 10^8$}
\end{subfigure}
\caption{Eigenvalues of a Uniform L-shaped domain} 
\label{fig:uniform_lshaped_different_lam}
\end{figure}%
\begin{figure}[t!]
\centering
\begin{subfigure}[b]{0.3\textwidth}
\centering
\includegraphics[width=\textwidth]{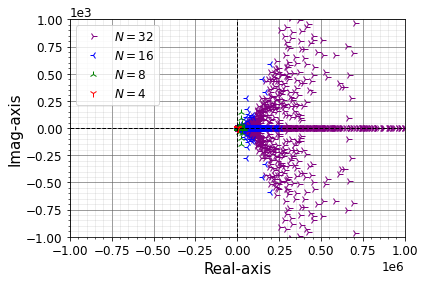}
\caption{$\lambda = 1$}
\end{subfigure}
\hfill
\begin{subfigure}[b]{0.3\textwidth}
\centering
\includegraphics[width=\textwidth]{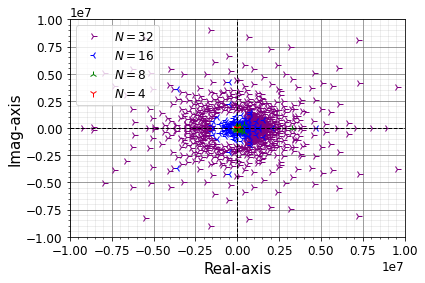}
\caption{$\lambda = 100$}
\end{subfigure}
\hfill
\begin{subfigure}[b]{0.3\textwidth}
\centering
\includegraphics[width=\textwidth]{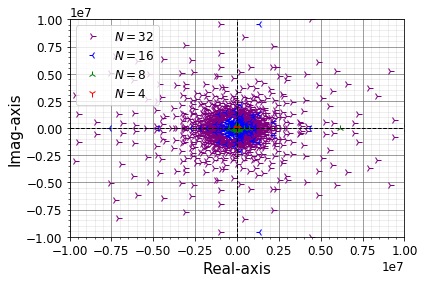}
\caption{$\lambda = 10^8$}
\end{subfigure}
\caption{Eigenvalues of Nonuniform L-shaped domain} 
\label{fig:nonuniformlshaped_different_lam}
\end{figure}%

For Uniform meshes, as Figure 
\ref{fig:uniform_lshaped_different_lam} shows, occurrence of
complex eigenvalues for $\lambda=1$ are more present than in the previous case.
Moreover, complex eigenvalues appear for higher values of $\lambda$ 
and grow when refining. The significant eigenvalues are those close to the origin
which are positive and real or close to being real in the limit. 

The Nonuniform mesh shows again a similar picturesque behavior as in the case
of the square domain.
Figure \ref{fig:nonuniformlshaped_different_lam} shows the distribution of
eigenvalues for this case.
Even for small $\lambda$ there are eigenvalues with large imaginary part in
this case, although with positive real part.

\begin{figure}[t!]
\begin{subfigure}[b]{0.45\textwidth}
\centering
\includegraphics[width=7cm]{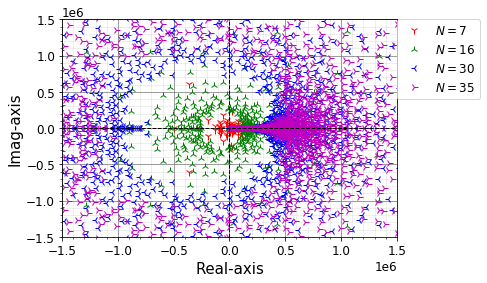}
\caption{Square domain}
\label{fig:nonuniform_square_mesh_withing_a_circle}
\end{subfigure}
\hfill
\begin{subfigure}[b]{0.45\textwidth}
\centering
\includegraphics[width=7cm]{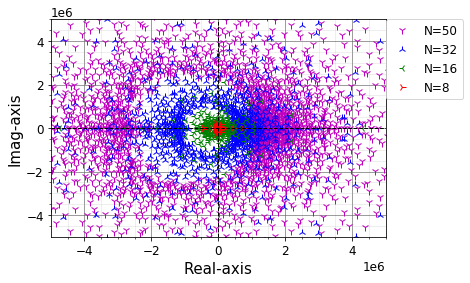}
\caption{L-shaped domain}
\label{fig:nonuniform_lshape_mesh_withing_a_circle}
\end{subfigure}
\caption{Eigenvalues within a circle of radius $R$ for Nonuniform mesh on a square and L-shaped domains}
\label{fig:nonuniform_mesh_withing_a_circle}
\end{figure}
An interesting observation, that matches the convergence of eigenvalues in 
Property~\ref{def:convergence_of_eigenvalues}, is clearly observed on Nonuniform meshes. 
To show this, we choose a Nonuniform mesh with $\lambda =10^8$ for both domains.
As seen in Figure \ref{fig:nonuniform_mesh_withing_a_circle}, when refining the mesh, 
eigenvalues are distributed allover the complex plane, 
within a circle of radius $R$ centered at the origin. 
The more we refine, the larger the circle becomes, indicating that the number of positive and finite
eigenvalues of interest are present inside that circle.
The circle can be taken larger and larger as the mesh size $h$ approaches zero.
%
\begin{figure}[t!]
\centering
\begin{subfigure}[b]{0.3\textwidth}
\includegraphics[width=\textwidth]{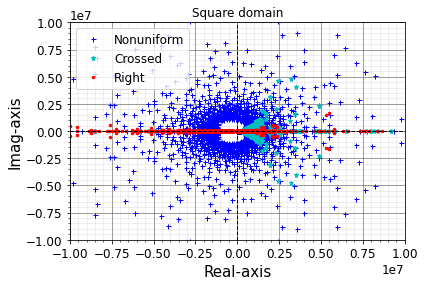}
\caption{Zoom out}
\label{fig:lam10pow8_square_allmeshes_outer_ZOOM}
\end{subfigure}
\hfill
\centering
\begin{subfigure}[b]{0.3\textwidth}
\centering
\includegraphics[width=\textwidth]{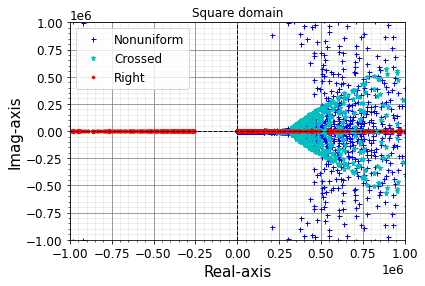}
\caption{Intermediate zoom}
\label{fig:lam10pow8_square_allmeshes_Intermediate_ZOOM}
\end{subfigure}
\hfill
\centering
\begin{subfigure}[b]{0.3\textwidth}
\centering
\includegraphics[width=\textwidth]{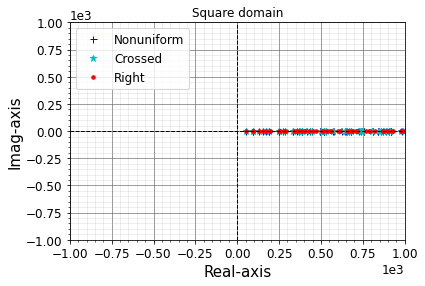}
\caption{Zoom in}
\label{fig:lam10pow8_square_allmeshes_ZOOMED}
\end{subfigure}
\hfill
\begin{subfigure}[b]{0.3\textwidth}
\centering
\includegraphics[width=\textwidth]{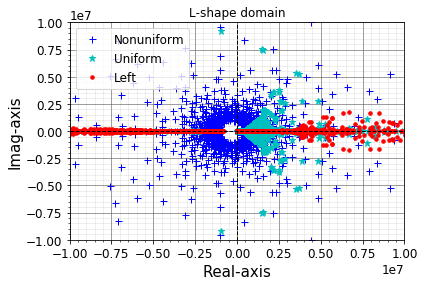}
\caption{Zoom out}
\label{fig:lam10pow8_lshape_allmeshes_outer_ZOOM}
\end{subfigure}
\hfill
\begin{subfigure}[b]{0.3\textwidth}
\centering
\includegraphics[width=\textwidth]{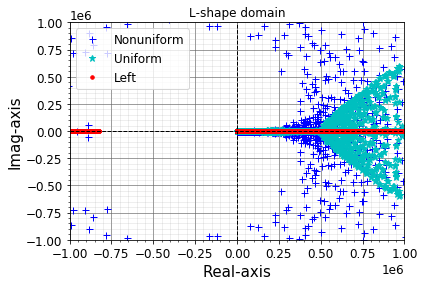}
\caption{Intermediate zoom}
\label{fig:lam10pow8_lshape_allmeshes_Intermediate_ZOOM}
\end{subfigure}
\hfill
\begin{subfigure}[b]{0.3\textwidth}
\centering
\includegraphics[width=\textwidth]{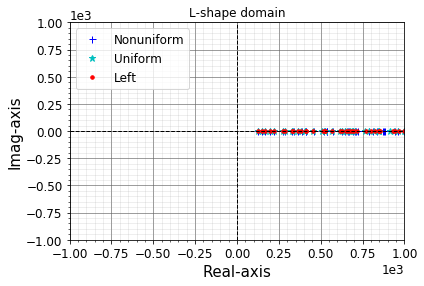}
\caption{Zoom in}
\label{fig:lam10pow8_lshape_allmeshes_ZOOMED}
\end{subfigure}
\caption{Eigenvalues for meshes zooming to region close to the origin}
\label{fig:Zoomed_fig_allMeshes}
\end{figure}

We conclude our numerical experiments with more pictures for different
values of the zoom parameters, so that the convergence behavior corresponding
to Property~\ref{def:convergence_of_eigenvalues} can be better appreciated.
Figure \ref{fig:Zoomed_fig_allMeshes} reports eigenvalues of both domains
with $\lambda=10^8$. The outer zoom lens in 
Figures~\ref{fig:lam10pow8_square_allmeshes_outer_ZOOM} 
and ~\ref{fig:lam10pow8_lshape_allmeshes_outer_ZOOM} gives a clear picture on the spread of eigenvalues in the complex plane. Looking at an intermediate zoom lens as~\ref{fig:lam10pow8_square_allmeshes_Intermediate_ZOOM} and
~\ref{fig:lam10pow8_lshape_allmeshes_Intermediate_ZOOM} show, the structure of eigenvalues for both domains is still not aliened with the region of interest as the spread is apparent. On the other hand, a zoomed in picture as Figure~\ref{fig:lam10pow8_square_allmeshes_ZOOMED} gives eigenvalues, on a square,  which are real and positive with respect to the previously reported pictures. Similar level of zoom show the same nice behavior in the case of the L-shaped domain
(see Figure~\ref{fig:lam10pow8_lshape_allmeshes_ZOOMED}). Thus, the scheme produces real and positive eigenvalues in the region of interest with no spurious modes as expected.

\section*{Acknowledgments}

FB gratefully acknowledges support by the Deutsche Forschungsgemeinschaft in
the Priority Program SPP 1748 \textit{Reliable simulation techniques in solid
mechanics, Development of non standard discretization methods, mechanical and
mathematical analysis} under the project number BE 6511/1-1.
DB is member of INdAM Research group GNCS and his research is partially
supported by PRIN/MIUR and by IMATI/CNR.

\end{document}